\documentclass[11pt]{article}
\usepackage{amssymb}
\usepackage{amsbsy}
\usepackage[latin1]{inputenc}
\usepackage{amsthm}
\usepackage[dvips]{graphicx}
\usepackage{graphicx} 
\usepackage{pst-eucl}
\usepackage[latin1]{inputenc}
\usepackage[english]{babel}
\usepackage{amsmath,amssymb,graphics,mathrsfs}
\usepackage{amsmath,amssymb,latexsym}
\usepackage{graphicx,color}
\usepackage[T1]{fontenc}
\usepackage[active]{srcltx}
\usepackage{multicol}
\usepackage[latin1]{inputenc}
\usepackage{pst-all}
\usepackage{enumerate}
\usepackage{pstricks}
\usepackage{pstricks-add}
\usepackage{setspace}
\usepackage{soul}
\usepackage{cancel}
\usepackage{nonfloat}
\usepackage[margin=10pt,font=footnotesize,labelfont=bf,labelsep=endash]{caption}
\usepackage[left=4cm,top=3cm,right=2.4cm,bottom=3.2cm]{geometry}
\usepackage[colorlinks=true,citecolor=red,linkcolor=blue,urlcolor=RubineRed,pdfpagetransition=Blinds,pdftoolbar=false,pdfmenubar=false]{hyperref}

\newcommand{\R}{\hbox{\rm I \kern -5pt R}}     
\newcommand{\p} {\hbox{\rm I \kern -5pt P}}

\def\x  {\boldsymbol x} 
\def\n  {\boldsymbol n} 

\def\H   {\boldsymbol H} 
\def\L    {\boldsymbol L} 

\def\Om       {\Omega}

\newtheorem{prop}{Proposition}[section]
\newtheorem{defi}[prop]{Definition}
\newtheorem{tma}[prop]{Theorem}
\newtheorem{cor}[prop]{Corollary}
\newtheorem{obs}[prop]{Remark}
\newtheorem{lem}[prop]{Lemma}

\allowdisplaybreaks

\begin{document}

\title{Study of a chemo-repulsion model with quadratic production. Part II: Analysis of an unconditionally energy-stable fully discrete scheme}
\author{F.~Guill\'en-Gonz\'alez\thanks{Dpto. Ecuaciones Diferenciales y An\'alisis Num\'erico and IMUS, 
Universidad de Sevilla, Facultad de Matem\'aticas, C/ Tarfia, S/N, 41012 Sevilla (SPAIN). Email: guillen@us.es, angeles@us.es},
M.~A.~Rodr\'{\i}guez-Bellido$^*$~and
 D.~A.~Rueda-G\'omez$^*$\thanks{Escuela de Matem\'aticas, Universidad Industrial de Santander, A.A. 678, Bucaramanga (COLOMBIA). Email:  diaruego@uis.edu.co}}

\date{}
\maketitle

\begin{abstract}

This work is devoted to the study  of a fully discrete scheme for a repulsive chemotaxis  with quadratic production model. By following the ideas presented in \cite{FMD}, we introduce an auxiliary variable (the gradient of the chemical concentration), and prove that the corresponding Finite Element (FE) backward Euler scheme is conservative and unconditionally energy-stable. Additionally, we also study some properties like solvability, a priori estimates, convergence towards weak solutions and error estimates. On the other hand, we propose two  linear iterative methods to approach the nonlinear scheme: an energy-stable Picard method and Newton's method. We prove solvability and  convergence of both methods towards the nonlinear scheme. Finally, we provide some numerical results in agreement with our theoretical analysis with respect to the error estimates.
\end{abstract}

\noindent{\bf 2010 Mathematics Subject Classification.} 35K51, 35Q92, 65M12, 65M15, 65M60, 92C17.

\noindent{\bf Keywords: } Chemorepulsion-production model, fully discrete scheme, finite element method, energy-stability, convergence, error estimates.

\section{Introduction}
The aim of this paper is to study an  unconditionally energy-stable fully discrete scheme for the following parabolic-parabolic repulsive-productive chemotaxis model (with quadratic production term):
\begin{equation}  \label{modelf00}
\left\{
\begin{array}
[c]{lll}%
\partial_t u - \Delta u = \nabla\cdot (u\nabla v)\ \ \mbox{in}\ \Omega,\ t>0,\\
\partial_t v - \Delta v + v =  u^{2} \ \mbox{in}\ \ \Omega,\ t>0,\\
\frac{\partial u}{\partial \n}=\frac{\partial v}{\partial \n}=0\ \ \mbox{on}\ \partial\Omega,\ t>0,\\
u(\x,0)=u_0(\x)\ge 0,\ v(\x,0)=v_0(\x)\ge 0\ \ \mbox{in}\ \Omega,
\end{array}
\right. \end{equation}
where $\Omega$ is a $n-$dimensional open bounded domain, $n=1,2,3$, with boundary $\partial \Omega$. The unknowns for this model are $u(\x, t) \ge 0$, the cell density, and $v(\x, t) \ge 0$, the chemical concentration. Problem (\ref{modelf00}) is conservative in $u$, because the total mass $\int_\Omega u(t)$ remains constant in time, as we can check  integrating equation (\ref{modelf00})$_1$ in $\Omega$, 
\begin{equation*}
\frac{d}{dt}\left(\int_\Omega u\right)=0, \ \ \mbox{ i.e. } \ 
\int_\Omega u(t)=\int_\Omega u_0 := m_0\vert \Omega\vert, \ \ \forall t>0.
\end{equation*}
In \cite{FMD} it was proved that  there exist  global in time ``weak-strong'' solutions of problem (\ref{modelf00}) in the following sense: $u\geq 0$ and $v\geq 0$ a.e.~$(t,\x)\in (0,+\infty)\times \Omega$,
\begin{equation}
\begin{array}{ccc}
(u-m_0,v-m_0^2) \in L^{\infty}(0,+\infty;L^2(\Omega)\times H^1(\Omega)) 
 \cap L^{2}(0,+\infty;H^1(\Omega) \times H^2(\Omega)), \\
( \partial_t u, \partial_t v) \in L^{q'}(0,T;H^1(\Omega)' \times L^2(\Omega)), \ \ \forall T>0,
\end{array} \label{wsa}
\end{equation}
where $q'=2$ in the 2-dimensional case (2D) and $q'=4/3$ in the 3-dimensional case (3D) ($q'$ is the conjugate exponent of $q=2$ in 2D and $q=4$ in 3D), satisfying the $u$-equation (\ref{modelf00})$_1$ in a variational sense, the $v$-equation (\ref{modelf00})$_2$ pointwisely a.e.~$(t,\x)\in (0,+\infty)\times \Omega$, and the following energy inequality a.e. $t_0,t_1:t_1\geq t_0\geq 0$:
\begin{equation*}
\mathcal{E}(u(t_1),v(t_1)) - \mathcal{E}(u(t_0),v(t_0))
 + \int_{t_0}^{t_1} (\Vert \nabla u(s)  \Vert_{L^2}^2 + \frac{1}{2} \Vert \Delta v(s) \Vert_{L^2}^2 +\frac{1}{2} \Vert \nabla v(s) \Vert_{L^2}^2)\ ds \leq0,
\end{equation*}
where $\mathcal{E}(u,v)=\frac{1}{2} \Vert u \Vert_{L^2}^2 
+ \frac{1}{4} \Vert \nabla v \Vert_{L^2}^2$. 
Moreover, assuming  the following regularity criterion:
\begin{equation*}
(u,\nabla v) \in  L^{\infty}(0,+\infty;H^1(\Omega)\times \H^1(\Omega)),
\end{equation*}
(which, at least is true in $1D$ and $2D$ domains), it was proved in \cite{FMD} that there exists a unique global in time strong solution of (\ref{modelf00}) satisfying  
\begin{equation}\label{kr3}
\left\{\begin{array}{rcl}
(u-m_0,v-m_0^2) &\in & L^{\infty}(0,+\infty; H^2(\Omega)^2) \cap L^2(0,+\infty;H^3(\Omega)^2),\\
(\partial_t u,\partial_t v) &\in & L^{\infty}(0,+\infty;L^2(\Omega)\times H^1(\Omega)) \cap L^2(0,+\infty;H^1(\Omega)\times H^2(\Omega)), \\
(\partial_{tt} u, \partial_{tt} v) &\in & L^2(0,+\infty;H^1(\Omega)' \times L^2(\Omega)).\end{array}\right.
\end{equation}
In particular, (\ref{kr3})$_1$ implies that $ (u,v)\in L^\infty(0,+\infty;L^\infty(\Om)^2)$. It should be desirable to design numerical methods for the model (\ref{modelf00}) conserving at the discrete level the main properties of the continuous model, such as mass-conservation, energy-stability, positivity and regularity. 

\

In relation to the study of chemo-repulsion models, there are some results about existence, uniqueness, regularity and qualitative properties of the solutions (\cite{C5:Cristian,Marcel,FMD,Yulin,T1,Tel}). In \cite{C5:Cristian}, the well-posedness of a chemo-repulsion model with linear production was studied, proving existence of global in time weak solutions and, for $2D$ domains, existence and  uniqueness of global in time strong solution. In the case of superlinear diffusion, global existence and uniqueness of solution in $nD$ domains (for $n\geq 3$) have been proved in \cite{Marcel}. Tao, in \cite{T1}, analyzed a chemo-repulsion model with nonlinear chemotactic sensitivity and linear production in $nD$ domains (with $n\geq 3$). Under some constraints  on the  chemotactic sensitivity function,  the existence of bounded classical solutions and the asymptotic convergence to the constant steady  state were proved. In \cite{Tel},  an extension of the Lotka-Volterra competition model was studied, in which a chemo-repulsive signal allows to one of the species to avoid encounters with rivals. The existence of global classical solution for the parabolic-parabolic and parabolic-elliptic cases in $nD$ domains (for $n\geq 1$) were proved there. In \cite{Yulin}, the existence,
uniform boundedness and long time behaviour of classical global solution were proved for a parabolic-elliptic chemo-repulsion system with nonlinear chemotactic sensitivity and nonlinear production. In  \cite{Win1}, radially symmetric solutions of a parabolic-elliptic chemoattraction system with nonlinear signal production ($u^p$) were studied, giving sufficient conditions (on the power $p$) under which  global bounded classical solution can be found.

\

On the other hand, some previous  works about numerical analysis for  chemotaxis models are the following. For the Keller-Segel system (i.e.~with chemo-attraction and linear production), in \cite{Filbet} Filbet studied  the existence of discrete solutions and the  convergence of a finite volume scheme. Saito, in \cite{Saito2}, proved error estimates for a conservative Finite Element (FE) approximation.   A mixed FE approximation  was studied in \cite{Marrocco}. In \cite{Eps}, some error estimates were proved for a fully discrete discontinuous FE method.  An energy-stable finite volume scheme for the Keller-Segel model with an additional cross-diffusion term has been studied in \cite{C5:BJ}. In \cite{Saito1}, a finite volume approximation for the parabolic-elliptic Keller-Segel system was studied, obtaining some error estimates and analyzing the blow-up phenomenon for the numerical
solution. The convergence of a characteristic splitting mixed finite element scheme for the Keller-Segel system was studied in \cite{Zhan-Zhu} and the corresponding error estimates were derived. In \cite{C5:FMD4}, unconditionally energy stable FE schemes for a chemo-repulsion model with linear production were studied. The convergence of a combined finite volume-nonconforming FE scheme was studied in \cite{CST}, in the case where the chemotaxis occurs in heterogeneous medium. In \cite{Fou}, the convergence of a positive nonlinear control volume finite element scheme for solving an anisotropic degenerate breast cancer development model (in which, chemotaxis phenomenon is included) was analyzed.

\

In this paper, we propose an unconditionally energy-stable fully discrete FE scheme, which inherit  some other  properties from the continuous model,  such as mass-conservation, and weak and strong estimates analogous to (\ref{wsa}) and (\ref{kr3}). Moreover, with respect to the positivity of the discrete variables $u^n_h$ and $v^n_h$, we can deduce that $v^n_h \geq 0$ (see Remark \ref{PoV}), but the positivity of discrete cell density $u^n_h$  can not be assured.

\

In order to design the scheme, we follow  the ideas presented in \cite{FMD}, where  (\ref{modelf00}) is reformulated by introducing the auxiliary  variable ${\boldsymbol\sigma}=\nabla v$ instead of $v$. Then,  model (\ref{modelf00}) is rewritten as:
\begin{equation}.  \label{modelf01}
 \left\{
\begin{array}
[c]{lll}%
\partial_t u - \nabla\cdot (\nabla u) = \nabla\cdot (u{\boldsymbol \sigma})\ \ \mbox{in}\ \Omega,\ t>0,\\
\partial_t {\boldsymbol \sigma} - \nabla(\nabla\cdot {\boldsymbol \sigma}) + {\boldsymbol \sigma} +\mbox{rot}(\mbox{rot }{\boldsymbol \sigma}) = \nabla(u^2) \ \mbox{in}\ \ \Omega,\ t>0,\\
\frac{\partial u}{\partial \n}=0\ \ \mbox{on}\ \partial\Omega,\ t>0,\\
{\boldsymbol \sigma}\cdot \n=0, \ \ \left[\mbox{rot }{\boldsymbol \sigma} \times \n\right]_{tang}=0 \ \ \mbox{on}\ \partial\Omega,\ t>0,\\
u(\x,0)=u_0(\x)\ge 0,\ {\boldsymbol \sigma}(\x,0)=\nabla v_0(
\x)\ \ \mbox{in}\ \Omega,
\end{array}
\right. 
\end{equation}
where (\ref{modelf01})$_2$ has been obtained by applying the gradient operator to equation (\ref{modelf00})$_2$  and adding the term $\mbox{rot}(\mbox{rot }{\boldsymbol \sigma})$ using that $\mbox{rot }{\boldsymbol \sigma}=\mbox{rot}(\nabla v)=0$. Once system (\ref{modelf01}) is solved, $v$ can be recovered from $u^2$ by solving
\begin{equation} \left\{
\begin{array}
[c]{lll}%
\partial_t v -\Delta v + v = u^2  \ \mbox{in}\ \ \Omega,\ t>0,\\
\frac{\partial v}{\partial \n}=0\ \ \mbox{on}\ \partial\Omega,\ t>0,\\
 v(\x,0)=v_0(\x)>0\ \ \mbox{in}\ \Omega.
\end{array}
\right.  \label{modelf01eqv}
\end{equation}
The outline of this paper is as follows: 
In Section 2,  the notation and some preliminary results are given. 
In Section 3,  the properties of the FE backward Euler scheme corresponding to formulation (\ref{modelf01})-(\ref{modelf01eqv}) are studied, 
including the  mass conservation, unconditional energy-stability, sol\-va\-bility, weak and strong estimates, convergence towards weak solutions, and optimal error estimates. 
In Section 4, two different linear iterative methods are proposed 
in order to approach the nonlinear scheme described in Section 3, which are an energy-stable Picard method and Newton's method. 
Solvability of these  methods and  convergence  towards the nonlinear scheme are also proved.
 Finally, in Section 5, some numerical results, in agreement with the theoretical analysis about the error estimates, are presented.

\section{Notations and preliminary results}
 The classical Sobolev spaces $H^m(\Omega)$ and Lebesgue spaces $L^p(\Omega)$, 
$1\leq p\leq \infty,$ with norms $\Vert\cdot\Vert_{m}$ and $\Vert\cdot \Vert_{L^p}$, respectively, will be considered.  
In particular,  the $L^2(\Omega)$-norm will be denoted by $\Vert
\cdot\Vert_0$.  The space $\H^{1}_{\sigma}(\Omega)$ is defined as $\H^{1}_{\sigma}(\Omega):=\{\mathbf{u}\in \H^{1}(\Omega): \mathbf{u}\cdot \n=0 \mbox{ on } \partial\Omega\}$ and  the following equivalent norms in $H^1(\Omega)$  and ${\bf H}_{\sigma}^1(\Omega)$, respectively (see \cite{necas} and \cite[Corollary 3.5]{Nour}, respectively) 
will be used:
\begin{equation*}
\Vert u \Vert_{1}^2=\Vert \nabla u\Vert_{0}^2 + \left( \int_\Omega u\right)^2, \ \ \forall u\in H^1(\Omega),
\end{equation*}
\begin{equation*}
\Vert {\boldsymbol\sigma} \Vert_{1}^2=\Vert {\boldsymbol\sigma}\Vert_{0}^2 + \Vert \mbox{rot }{\boldsymbol\sigma}\Vert_0^2 + \Vert \nabla \cdot {\boldsymbol\sigma}\Vert_0^2, \ \ \forall {\boldsymbol\sigma}\in \H^{1}_{\sigma}(\Omega).
\end{equation*}
If $Z$ is a
general Banach space, its topological dual will be denoted by $Z'$.
Moreover, the letters $C,C_i,K_i$ will denote different positive
constants independent of discrete parameters.
 
 \
 
The following linear elliptic operators are introduced, namely
	\begin{equation}\label{B001} 
	\widehat A u =g \quad \Longleftrightarrow \quad 
	\left\{\begin{array}{l}
	-\Delta u+ \int_\Omega u = g \ \mbox{ in } \Omega,\\
	\displaystyle\frac{\partial u}{\partial\n}=0 \ \mbox{ on }  \partial\Omega, \vspace{0,2 cm}\\
	\end{array}\right.
	\end{equation} 
\begin{equation}\label{B101} 
A v =g \quad \Longleftrightarrow \quad 
\left\{\begin{array}{l}
-\Delta v + v = g \ \mbox{ in } \Omega,\\
\displaystyle\frac{\partial v}{\partial\n}=0 \ \mbox{ on }  \partial\Omega, \vspace{0,2 cm}\\
\end{array}\right.
\end{equation} 
and 
\begin{equation} \label{B201}
B{\boldsymbol \sigma} = h \quad \Longleftrightarrow \quad 
\left\{\begin{array}{l}
-\nabla(\nabla\cdot {\boldsymbol \sigma}) + \mbox{rot(rot }{\boldsymbol \sigma}\mbox{)} + {\boldsymbol \sigma}  = h \ \mbox{ in }  \Omega,\\
{\boldsymbol \sigma}\cdot \n=0, \ \ \left[\mbox{rot }{\boldsymbol \sigma} \times \n\right]_{tang}=0 \ \mbox{ on }  \partial\Omega,
\end{array}\right. 
\end{equation}
which, in variational form, are given by 
$\widehat{A},A:H^1(\Omega)\rightarrow H^1(\Omega)'$ and $B:\H^{1}_{\sigma}(\Omega)\rightarrow \H^{1}_{\sigma}(\Omega)'$ such that 
	\begin{equation*}
	\langle \widehat A u,\bar{u}\rangle=(\nabla u, \nabla \bar{u})+\left(\int_\Omega u\right) \left(\int_\Omega\bar{u}\right), \ \ \forall u,\bar{u}\in {H}^{1}(\Omega),
	\end{equation*}
\begin{equation*}
\langle Av ,\bar{v}\rangle=(\nabla v, \nabla \bar{v}) + (v, \bar{v}), \ \ \forall v,\bar{v}\in {H}^{1}(\Omega),
\end{equation*}
\begin{equation*}
\langle B{\boldsymbol\sigma},\bar{{\boldsymbol\sigma}}\rangle=({\boldsymbol\sigma},\bar{{\boldsymbol\sigma}})+(\nabla \cdot {\boldsymbol\sigma}, \nabla\cdot \bar{\boldsymbol\sigma})+(\mbox{rot }{\boldsymbol \sigma},\mbox{rot }\bar{\boldsymbol \sigma}),\ \ \forall {\boldsymbol\sigma},\bar{{\boldsymbol\sigma}}\in
\H^{1}_{\sigma}(\Omega).
\end{equation*} 
The $H^2$-regularity of problems (\ref{B001})-(\ref{B201}) must be assumed. 
Consequently, there exist some constants $C>0$ such that
\begin{equation}\label{H2us1-a}
 \Vert u\Vert_{2}\leq C\Vert \widehat{A} u \Vert_{0} \ \ \ \forall u\in H^2(\Omega), \ \  \Vert v\Vert_{2}\leq C\Vert A v \Vert_{0} \ \ \ \forall v\in H^2(\Omega),
\end{equation}
\begin{equation} \label{H2us1-b}
\Vert {\boldsymbol\sigma}\Vert_{2}\leq C\Vert B {\boldsymbol\sigma} \Vert_{0} \ \ \ \forall {\boldsymbol\sigma}\in \H^2(\Omega). 
\end{equation}
The classical 3D interpolation inequality will be repeatedly used
\begin{equation}\label{in3D}
\Vert u\Vert_{L^3}\leq C\Vert u\Vert_0^{1/2}\Vert u\Vert_{1}^{1/2} \  \ \forall u\in H^1(\Omega).
\end{equation}
Finally,  the following result will also be used (see
\cite{Shen}):
\begin{lem}{\bf (Uniform discrete Gronwall lemma)}\label{LGD001}
 Let $k>0$ and $d^n,g^n,h^n\geq 0$ such that 
\begin{equation*}
\frac{d^{n+1}-d^n}{k} \leq   g^n d^n +  h^n, \ \ \forall n\geq 0.
\end{equation*}
If for any $r\in \mathbb{N}$, there exist $a_1(t_r)$, $a_2(t_r)$ and $a_3(t_r)$ depending on $t_r =kr$, such that
\begin{equation*}\label{Gb03}
k\underset{n=n_0}{\overset{n_0+r-1}{\sum}} g^n\leq a_1(t_r), \ \ k\underset{n=n_0}{\overset{n_0+r-1}{\sum}} h^n\leq a_2(t_r), \ \ k\underset{n=n_0}{\overset{n_0+r-1}{\sum}} d^n\leq a_3(t_r),
\end{equation*}
for any integer $n_0\geq 0$, 
then
\begin{equation*}\label{Gb04}
d^n\leq \left(a_2(t_r)+ \frac{a_3(t_r)}{t_r} \right) \mbox{exp}\left\{ a_1(t_r)\right\}, \ \ \forall n\geq r.
\end{equation*}
\end{lem}

As  a consequence of Lemma \ref{LGD001} and the classical discrete Gronwall Lemma, 
 the following result holds (see \cite[Corollary 2.4.]{FMD}):
\begin{cor}\label{CorUnif}
 Assume conditions of Lemma \ref{LGD001}. Let $k_0\in \mathbb{N}$ be fixed, then the following estimate holds for all $k\leq k_0$
\begin{equation*}
d^n \leq C(d^0,k_0) \ \ \ \forall n\geq 0.
\end{equation*}
\end{cor}

\section{Fully discrete backward Euler scheme in the variables $(u,\boldsymbol\sigma)$}
This section is devoted to design an unconditionally energy-stable scheme for model (\ref{modelf00}) (with respect to a modified energy in the variables $(u,{\boldsymbol\sigma})$), using a FE discretization in space and the backward
Euler discretization in time (considered for simplicity on a uniform partition of $[0,+\infty)$ given by $t_n=nk$, where $k>0$ denotes the time step). 
Concerning the space discretization, let $\{\mathcal{T}_h\}_{h>0}$ be a family of shape-regular and quasi-uniform  triangulations of $\overline{\Omega}$ made up of simplexes $K$ (triangles in two dimensions and tetrahedra in three dimensions), such that $\overline{\Omega}= \cup_{K\in \mathcal{T}_h} K$, where $h = \max_{K\in \mathcal{T}_h} h_K$, with $h_K$ being the diameter of $K$. Furthermore, $\mathcal{N}_h = \{ \mathbf{a}_{i}\}_{i\in \mathcal{I}}$ denotes the set of all nodes of $\mathcal{T}_h$. 
 The following continuous FE spaces for $u$, ${\boldsymbol\sigma}$ and $v$, are chosen:
$$(U_h, {\boldsymbol\Sigma}_h, V_h) \subset H^1 \times \H^1_{\sigma}\times W^{1,6}\quad \hbox{generated by $\mathbb{P}_k,\mathbb{P}_m,\mathbb{P}_{r}$ with $k,m,r\geq 1$.}
$$

Now, the linear operators  $\widehat{A}_h: H^1(\Omega)\rightarrow U_h$, $B_h: \H_{\sigma}^1(\Omega) \rightarrow {\boldsymbol\Sigma}_h$ and ${A}_h: H^1(\Omega)\rightarrow V_h$ are considered, defined by:
\begin{equation}\label{opnew1}
\begin{array}{lll}

\displaystyle(\widehat{A}_h u_h,\bar{u}_h)=(\nabla u_h,\nabla \bar{u}_h) + \left(\int_{\Omega} u_h\right) \left( \int_{\Omega}\bar{u}_h\right), \ \ \forall \bar{u}_h\in U_h,  \\
(B_h {\boldsymbol\sigma}_h, \bar{\boldsymbol\sigma}_h)=(\nabla \cdot {\boldsymbol\sigma}_h,\nabla \cdot \bar{\boldsymbol\sigma}_h) + (\mbox{rot }{\boldsymbol\sigma}_h,\mbox{rot }\bar{\boldsymbol\sigma}_h) + ({\boldsymbol\sigma}_h,\bar{\boldsymbol\sigma}_h), \ \ \forall \bar{\boldsymbol\sigma}_h \in {\boldsymbol\Sigma}_h,\\
({A}_h v_h, \bar{v}_h)=(\nabla v_h,\nabla \bar{v}_h)  + ({v}_h,\bar{v}_h), \ \ \forall \bar{v}_h \in {V}_h.
\end{array}
\end{equation}
Moreover, we choose the following interpolation operators:  
$$\mathcal{R}_h^u:H^1(\Omega)\rightarrow U_h,\quad \mathcal{R}_h^{\boldsymbol \sigma}:\H^1_{\sigma}(\Omega)\rightarrow {\boldsymbol \Sigma}_h,\quad \mathcal{R}_h^v:H^1(\Omega)\rightarrow V_h,
$$
  such that, for all $u\in H^1(\Omega)$, ${\boldsymbol\sigma}\in \H^1_\sigma(\Omega)$ and $v\in H^1(\Omega)$, the operators $\mathcal{R}_h^u u\in U_h$, $\mathcal{R}_h^{\boldsymbol\sigma} {\boldsymbol\sigma}\in {\boldsymbol\Sigma}_h$ and $\mathcal{R}_h^v v\in V_h$ satisfy respectively
\begin{equation}\label{interp2}
(\nabla
(\mathcal{R}_h^u u - u), \nabla\bar{u}_h)   + \left(\int_{\Omega} (\mathcal{R}_h^u u - u) \right) \left( \int_{\Omega}\bar{u}_h\right)=0,\ \ \forall \bar{u}_h\in U_h,
\end{equation}
\begin{equation}\label{interp1}
(\nabla \cdot (\mathcal{R}_h^{\boldsymbol\sigma} {\boldsymbol\sigma} - {\boldsymbol\sigma}),\nabla \cdot \bar{\boldsymbol\sigma}_h) + (\mbox{rot}(\mathcal{R}_h^{\boldsymbol\sigma} {\boldsymbol\sigma}-{\boldsymbol\sigma}),\mbox{rot }\bar{\boldsymbol\sigma}_h) + (\mathcal{R}_h^{\boldsymbol\sigma} {\boldsymbol\sigma} - {\boldsymbol\sigma},\bar{\boldsymbol\sigma}_h)=0,\ \ \forall \bar{\boldsymbol
\sigma}_h\in {\boldsymbol\Sigma}_h,
\end{equation}
\begin{equation}\label{interp2a}
(\nabla (\mathcal{R}_h^v v - v),\nabla \bar{v}_h) + (\mathcal{R}_h^v v - v, \bar{v}_h)=0,\ \ \forall \bar{v}_h\in V_h .
\end{equation}
 Observe that, from Lax-Milgram Theorem, the interpolation operators $\mathcal{R}_h^u$, $\mathcal{R}_h^{\boldsymbol \sigma}$ and $\mathcal{R}_h^v$ are well defined. Moreover, the following interpolation errors hold
\begin{equation}\label{EI01}
\frac{1}{h} \Vert \mathcal{R}_h^u u - u\Vert_0 + \Vert \mathcal{R}_h^u u - u\Vert_{1} \leq Ch^{k'} \Vert u\Vert_{k'+1} \ \ \forall u\in  H^{k'+1}(\Omega), \ \ (1\le k'\le k)
\end{equation}
\begin{equation}\label{EI02}
\frac{1}{h} \Vert \mathcal{R}_h^{\boldsymbol\sigma} {\boldsymbol\sigma} - {\boldsymbol\sigma}\Vert_0 + \Vert \mathcal{R}_h^{\boldsymbol\sigma} {\boldsymbol\sigma} - {\boldsymbol\sigma}\Vert_{1} \leq Ch^{m'} \Vert  {\boldsymbol\sigma}\Vert_{m'+1} \ \ \forall  {\boldsymbol\sigma}\in  \H^{m'+1}(\Omega),\ \ (1\le m'\le m)
\end{equation}
\begin{equation}\label{EI03}
\frac{1}{h} \Vert \mathcal{R}_h^v v - v\Vert_0 + \Vert \mathcal{R}_h^v v - v\Vert_{1} \leq Ch^{r'} \Vert v\Vert_{r'+1} \ \ \forall v\in  H^{r'+1}(\Omega), \ \ (1\le r'\le r).
\end{equation}
Also, the following stability property will be used 
\begin{equation}\label{SP01}
\Vert (\mathcal{R}_h^u u,\mathcal{R}_h^{\boldsymbol\sigma} {\boldsymbol\sigma},\mathcal{R}_h^v v) \Vert_{W^{1,6}}\leq C\Vert (u,{\boldsymbol\sigma},v)\Vert_2,
\end{equation}
which can be obtained from (\ref{EI01})-(\ref{EI03}), using the inverse inequality 
\begin{equation}\label{16inv1} 
\Vert (u_h,{\boldsymbol\sigma}_h,v_h)\Vert_{W^{1,6}}\leq Ch^{-1}\Vert (u_h,{\boldsymbol\sigma}_h,v_h)\Vert_1 \ \  \ \forall(u_h,{\boldsymbol\sigma}_h,v_h)\in U_h \times {\boldsymbol\Sigma}_h\times V_h,
\end{equation}
 and comparing $\mathcal{R}_h^{u,{\boldsymbol\sigma},v}$ with an average interpolation of Clement or Scott-Zhang type (which are stable in the $W^{1,6}$-norm).

\begin{lem}
Assume the $H^2$-regularity for problems (\ref{B001})-(\ref{B201}) given in (\ref{H2us1-a})-(\ref{H2us1-b}). Then, 
\begin{equation}\label{pd0a}
\Vert u_h\Vert_{W^{1,6}}\leq C\Vert \widehat{A}_h u_h\Vert_{0} \ \ \forall u_h\in U_h, \ \ \ \ 
\Vert {v}_h\Vert_{W^{1,6}}\leq C\Vert {A}_h {v}_h\Vert_{0} \ \ \forall v_h\in V_h,
\end{equation}
\begin{equation}\label{pd0b}
\Vert {\boldsymbol \sigma}_h\Vert_{W^{1,6}}\leq C\Vert B_h {\boldsymbol \sigma}_h\Vert_{0} \ \ \forall {\boldsymbol \sigma}_h\in {\boldsymbol \Sigma}_h.
\end{equation}
\end{lem}
\begin{proof}
First, we consider regular functions associated to the discrete functions $\widehat{A}_h u_h$, ${A}_h v_h$ and $B_h {\boldsymbol \sigma}_h$. We define $u(h), v(h)\in H^2(\Omega)$ and ${\boldsymbol \sigma}(h)\in \H^2(\Omega)$ as the solutions of elliptic problems 
$$
 A u(h)=\widehat{A}_h u_h,  \quad A v(h)=A_h v_h\quad \hbox{and}\quad B {\boldsymbol \sigma}(h)=B_h {\boldsymbol \sigma}_h. $$
In particular, from (\ref{H2us1-a})-(\ref{H2us1-b}), 
\begin{equation}\label{e0ppp}
 \Vert u(h)\Vert_{2}\leq C\Vert \widehat{A}_h u_h\Vert_{0}, \ \   \Vert v(h)\Vert_{2}\leq C\Vert {A}_h {v}_h\Vert_{0}  \ \ \mbox{ and } \ \ \Vert {\boldsymbol \sigma}(h)\Vert_{2}\leq C\Vert B_h {\boldsymbol \sigma}_h\Vert_{0}.
\end{equation}
We are going to prove \eqref{pd0b}, because \eqref{pd0a} can be proved analogously. 
Now, by applying \eqref{SP01} and  \eqref{16inv1}, we decompose the $W^{1,6}$-norm as:
\begin{eqnarray}\label{e0p2} 
\Vert {\boldsymbol \sigma}_h\Vert_{W^{1,6}}
&\leq& \Vert {\boldsymbol \sigma}_h - \mathcal{R}^{\boldsymbol \sigma}_h {\boldsymbol \sigma}(h) \Vert_{W^{1,6}} 
+ \Vert \mathcal{R}^{\boldsymbol \sigma}_h {\boldsymbol \sigma}(h) \Vert_{W^{1,6}}
\nonumber
\\ 
&\leq& C\, h^{-1} \Vert {\boldsymbol \sigma}_h - \mathcal{R}^{\boldsymbol \sigma}_h {\boldsymbol \sigma}(h) \Vert_{1} 
+C\, \Vert {\boldsymbol \sigma}(h) \Vert_{W^{1,6}}.
\end{eqnarray}
By testing $B{\boldsymbol\sigma}(h)$ by any $\bar{\boldsymbol\sigma}_h\in {\boldsymbol\Sigma}_h$ and using (\ref{opnew1})$_2$ we have
\begin{eqnarray}\label{ne00a1}
&(\nabla \cdot {\boldsymbol\sigma}_h,\nabla \cdot \bar{\boldsymbol\sigma}_h) &\!\!\!\!+ (\mbox{rot }{\boldsymbol\sigma}_h ,\mbox{rot }\bar{\boldsymbol\sigma}_h) + ({\boldsymbol\sigma}_h,\bar{\boldsymbol\sigma}_h)\nonumber\\
&&\!\!\!\!=(\nabla \cdot {\boldsymbol\sigma}(h),\nabla \cdot \bar{\boldsymbol\sigma}_h) + (\mbox{rot }{\boldsymbol\sigma}(h),\mbox{rot }\bar{\boldsymbol\sigma}_h) + ({\boldsymbol\sigma}(h),\bar{\boldsymbol\sigma}_h), \ \ \forall \bar{\boldsymbol\sigma}_h\in {\boldsymbol\Sigma}_h.
\end{eqnarray}
By subtracting at both sides of  (\ref{ne00a1}) the terms $(\nabla \cdot \mathcal{R}^{\boldsymbol\sigma}_h {\boldsymbol\sigma}(h), \nabla \cdot \bar{\boldsymbol\sigma}_h)$, $(\mbox{rot}\mathcal{R}^{\boldsymbol\sigma}_h {\boldsymbol\sigma}(h), \mbox{rot } \bar{\boldsymbol\sigma}_h)$ and $(\mathcal{R}^{\boldsymbol\sigma}_h {\boldsymbol\sigma}(h),\bar{\boldsymbol\sigma}_h)$, taking  $\bar{\boldsymbol\sigma}_h= {\boldsymbol\sigma}_h - \mathcal{R}^{\boldsymbol\sigma}_h {\boldsymbol\sigma}(h)\in {\boldsymbol\Sigma}_h$ in (\ref{ne00a1}), and using the H\"older inequality, 
\begin{equation}\label{D0a}
\Vert {\boldsymbol\sigma}_h - \mathcal{R}^{\boldsymbol\sigma}_h {\boldsymbol\sigma}(h)\Vert_1\leq C\Vert \mathcal{R}^{\boldsymbol\sigma}_h {\boldsymbol\sigma}(h)-{\boldsymbol\sigma}(h)\Vert_1\leq C h \Vert {\boldsymbol\sigma}(h)\Vert_{2},
\end{equation}
where the interpolation error (\ref{EI02}) was used in the last inequality. Finally, using (\ref{e0ppp}), (\ref{e0p2}) and (\ref{D0a}),  inequality (\ref{pd0b}) is deduced. 
\end{proof}

\subsection{Definition of the scheme \textbf{US}}\label{scheme-def}

Taking into account the reformulation (\ref{modelf01}), we consider the following FE backward Euler scheme in the  variables $(u,\boldsymbol\sigma)$ (\emph{Scheme \textbf{US}}, from now on) which is a  first order in time, nonlinear and coupled scheme (hereafter, we denote $\delta_t a^n=
(a^n - a^{n-1})/k$):
\begin{itemize}
\item{\bf Initialization}: We fix $(u^{0}_h,{\boldsymbol \sigma}^{0}_h) =(\mathcal{R}_h^u u_0, \mathcal{R}_h^{\boldsymbol\sigma} (\nabla v_0)) \in  U_h\times {\boldsymbol\Sigma}_h$ and $v^0_h=\mathcal{R}_h^v v_0\in V_h$. 
\item{\bf Time step} n: Given $(u^{n-1}_h,{\boldsymbol \sigma}^{n-1}_h)\in  U_h\times {\boldsymbol\Sigma}_h$, compute $(u^{n}_h,{\boldsymbol \sigma}^{n}_h)\in  U_h\times {\boldsymbol\Sigma}_h$ solving
\begin{equation}
\left\{
\begin{array}
[c]{lll}%
(\delta_t u^n_h,\bar{u}_h) + (\nabla u^n_h, \nabla \bar{u}_h) +(u^n_h{\boldsymbol \sigma}^n_h,\nabla \bar{u}_h)=0, \ \ \forall \bar{u}_h\in U_h,\\
(\delta_t {\boldsymbol \sigma}^n_h,\bar{\boldsymbol \sigma}_h) +
( B_h {\boldsymbol \sigma}^n_h, \bar{\boldsymbol
\sigma}_h)  -
2(u^n_h\nabla u^n_h,\bar{\boldsymbol \sigma}_h)=0,\ \ \forall
\bar{\boldsymbol \sigma}_h\in {\boldsymbol\Sigma}_h.
\end{array}
\right.  \label{modelf02}
\end{equation}
\end{itemize}

Once the scheme \textbf{US} is solved,  
$v^n_h=v^n_h((u^n_h)^2) \in V_h$ can be recovered by solving: 
\begin{equation}\label{edovf}
(\delta_t v^n_h,\bar{v}_h)  +( {A}_h v^n_h,\bar{v}_h)  =((u^n_h)^2,\bar{v}_h), \ \ \forall
\bar{v}_h\in V_h .
\end{equation}
Lax-Milgram theorem implies that there exists a unique $v^n_h \in V_h$ solution of (\ref{edovf}).

\begin{obs}\label{PoV}
By using the mass-lumping technique  in all terms of (\ref{edovf}) excepting the self-diffusion term $(\nabla v^n_h,\nabla\bar{v}_h)$, approximating by $\mathbb{P}_1$-continuous FE and imposing an acute triangulation (all angles of the triangles or tetrahedra must be at most $\pi/2$), one has  that if $v^{n-1}_h\geq 0$ then $v^n_h\geq 0$. However, at least in all numerical simulations that we have made without using mass-lumping, we have not found any example in which, starting with $v^0_h\geq 0$ we obtain $v^n_h(\mathbf{a}_i)<0$, for some $n>0$ and $\mathbf{a}_i$. 
\end{obs}

\subsection{Conservation, Solvability, Energy-Stability and Convergence}

Assuming that the functions $\bar{u}_h=1\in U_h$ and $\bar{v}_h=1\in V_h$,  one can deduce that the scheme \textbf{US} conserves in time the total mass $\int_\Omega u^n_h$, that is,
\begin{equation*}
\int_\Omega u^n_h=\int_\Omega u^{n-1}_h=\cdot\cdot\cdot=\int_\Omega
u^{0}_h,
\end{equation*}
and  the following behavior of $\int_\Omega v^n_h$ holds: 
\begin{equation*} 
\delta_t \left(\int_\Omega v^n_h \right)= \int_\Omega (u^n_h)^2  - \int_\Omega v^n_h.
\end{equation*}
Now, we establish some results concerning to the solvability and energy-stability of the scheme~\textbf{US}, but we will omit their proofs because those follow the same ideas given in \cite{FMD} (Theorem 4.4 and Lemma 4.7, respectively).
\begin{tma} {\bf(Unconditional existence and conditional uniqueness)} \label{USus}
There exists \linebreak{$(u^n_h,{\boldsymbol\sigma}^n_h) \in  U_h\times {\boldsymbol\Sigma}_h$} solution of the scheme \textbf{US}. Moreover, if 
\begin{equation}\label{uniq01}
k\Vert (u^n_h,{\boldsymbol \sigma}^n_h)
\Vert_{1}^4 \quad \hbox{is small enough,}
\end{equation}
then the solution is unique.
\end{tma}
\begin{obs}
In the case of 2D domains, since one has  estimate (\ref{strong01}) below, then the uniqueness restriction (\ref{uniq01}) can be relaxed to $k K_0^2$ small enough, where $K_0$ is a constant depending on data $(\Omega,u_0,{\boldsymbol\sigma}_0)$, but independent of $(k,h)$ and $n$. 
\end{obs}
\begin{obs}\label{inddom01}
In 3D domains, using the inverse inequality $\Vert u_h\Vert_{1}\leq \frac{C}{h}\Vert u_h\Vert_{0}$ (see Lemma 4.5.3 in \cite{Brenner}, p. 111) and estimate (\ref{weak01}) below, we have that 
$$
\Vert (u^n_h,{\boldsymbol \sigma}^n_h)
\Vert_{1}^4\leq \frac{C}{h^4} \Vert (u^n_h,{\boldsymbol \sigma}^n_h)
\Vert_0^4\leq \frac{C}{h^4} C_0^2,
$$
and therefore, the uniqueness restriction (\ref{uniq01}) can be rewritten as 
\begin{equation}\label{uua}
\frac{k }{h^4} \ \mbox{ small enough}.
\end{equation}
\end{obs}

\begin{defi}\label{enesf}
A numerical scheme with solution $(u_n,{\boldsymbol \sigma}_n)$ is called energy-stable with respect to the energy 
\begin{equation*}
\mathcal{E}(u,{\boldsymbol\sigma})= \frac{1}{2}\Vert u\Vert_0^2 + \frac{1}{4}\Vert {\boldsymbol\sigma}\Vert_0^2,
\end{equation*}
 if this energy is time decreasing, that is 
\begin{equation*}
\mathcal{E}(u^n_h,{\boldsymbol \sigma}^n_h)\leq \mathcal{E}(u^{n-1}_h,{\boldsymbol
\sigma}^{n-1}_h), \ \ \forall n.
\end{equation*}
\end{defi}

\begin{lem} {\bf (Unconditional energy-stability)} \label{estinc1}
The scheme \textbf{US} is unconditionally energy-stable with res\-pect to $\mathcal{E}(u,{\boldsymbol\sigma})$. In fact, for any $(u^n_h,{\boldsymbol
\sigma}^n_h)$ solution of the scheme \textbf{US}, the following discrete energy law holds
\begin{eqnarray}\label{lawenerfydisce}
&\delta_t \mathcal{E}(u^n_h,{\boldsymbol \sigma}^n_h)&\!\!\!\!\! +
\frac{k}{2} 
\Vert \delta_t u^n_h\Vert_{0}^2 + \frac{k}{4} \Vert \delta_t {\boldsymbol \sigma}^n_h\Vert_{0}^2+ \Vert \nabla u^n_h\Vert_{0}^{2} +
\displaystyle\frac{1}{2}\Vert  {\boldsymbol
\sigma}^n_h\Vert_{1}^{2}=0.
\end{eqnarray}
\end{lem}

\begin{obs}
Looking at (\ref{lawenerfydisce}), one can say that  scheme \textbf{US} introduces the following two first order ``numerical dissipation" terms: 
$$\frac{k}{2} 
\Vert \delta_t u^n_h\Vert_{0}^2\quad \hbox{and}  \quad \frac{k}{4} \Vert \delta_t {\boldsymbol \sigma}^n_h\Vert_{0}^2.
$$
\end{obs}
\subsubsection{Uniform weak estimates}
Starting from the (local in time) discrete energy law (\ref{lawenerfydisce}), 
 some global in time estimates for $(u^n_h,{\boldsymbol
\sigma}^n_h)$ 
will be obtained. The letters  $C,C_i,K_i$ denote different positive constants depending on the data $(\Omega,u_0, v_0)$, but independent of discrete parameters $(k, h)$ and time step $n$. Hereafter, in order to abbreviate, we introduce the notation: 
$$(\hat u, \hat v)= (u-m_0,v-m_0^2).$$
\vspace{-0.7 cm}
\begin{tma} \label{welem} {\bf(Weak estimates of $(u^n_h,{\boldsymbol \sigma}^n_h)$)}
Let $(u^n_h,{\boldsymbol \sigma}^n_h)$ be a solution of the scheme \textbf{US}. Then, the following estimates hold
\begin{equation}\label{weak01}
\Vert (u^n_h, {\boldsymbol
\sigma}^n_h)\Vert_{0}^{2}
+ k \underset{m=1}{\overset{n}{\sum}}\Vert (\hat{u}^m_h, {\boldsymbol \sigma}^m_h)\Vert_{1}^2\leq C_0, \ \ \ \forall n\geq 1.
\end{equation}
\end{tma}
\begin{proof}
The proof follows as in Theorem 4.9 of \cite{FMD}.
\end{proof}

 In contrast to what happens in the time-discrete scheme corresponding to \textbf{US} (see \cite{FMD}),
 in the fully discrete scheme \textbf{US} it is not clear how to quantify the relation ${\boldsymbol\sigma}^n_h\simeq \nabla v^n_h$. Therefore, the uniform estimates for $v^n_h$ can not be obtained directly from the estimates for ${\boldsymbol\sigma}^n_h$. Alternatively,  uniform weak estimates for $v^n_h$ will be directly obtained from (\ref{edovf}).\\
\begin{lem} {\bf (Weak estimates for $v^n_h$)}
Let $v^n_h$ be the solution of (\ref{edovf}). Then, the following estimate holds
\begin{equation}\label{weak02UVlinvL2}
\Vert v^n_h \Vert_0^2 + k \underset{m=1}{\overset{n}{\sum}}\Vert \hat{v}^m_h \Vert_1^2  \leq K_0, \ \ \ \forall n\geq 1.
\end{equation}
\end{lem}
\begin{proof}
Rewriting (\ref{edovf}) as
\begin{equation}\label{edovf-nnn}
	(\delta_t \hat{v}^n_h,\bar{v}_h)  +( {A}_h \hat {v}^n_h,\bar{v}_h)  =((\hat{u}^n_h+2 m_0)\hat{u}^n_h,\bar{v}_h), \ \ \forall
	\bar{v}_h\in V_h,
\end{equation}
and taking $\bar{v}=\hat{v}^n_h$ in (\ref{edovf-nnn}) one has
\begin{eqnarray*}
&\displaystyle\delta_t \Vert \hat{v}^n_h \Vert_0^2 &\!\!\!\!  
+ \Vert \hat{v}^n_h\Vert_1^2\leq C\Vert \hat{u}^n_h + 2 m_0\Vert_{L^{3/2}}^2 \Vert \hat{u}^n_h\Vert_{L^6}^2\leq C \Vert \hat{u}^n_h\Vert_{H^1}^2,
\end{eqnarray*}
from which, adding for $m=1,\cdot \cdot\cdot,n$ and using (\ref{weak01}), one can deduce (\ref{weak02UVlinvL2}).
\end{proof}

\subsubsection{Convergence}
Starting from the previous stability estimates, proceeding as in Theorem 4.11 of \cite{FMD}, the convergence  of the scheme \textbf{US} towards weak solutions as $(k,h)\rightarrow 0$ can be proved. 
Concretely, by introducing the functions:
\begin{itemize}
\item $(\widetilde{u}_{h,k},\widetilde{\boldsymbol\sigma}_{h,k})$ are continuous functions on $[0,+\infty)$, linear on each interval $(t_{n-1},t_n)$ and equal to $(u^n_h,{\boldsymbol\sigma}^n_h)$ at $t=t_n$, $n\geq 0$;
\item $({u}_{h,k},{\boldsymbol\sigma}_{h,k})$ are the piecewise constant functions taking values $(u^{n}_h,{\boldsymbol\sigma}^n_h)$ on $(t_{n-1},t_n]$, $n\geq 1$,
\end{itemize}
then,  the following result holds: 
\begin{tma} {\bf (Convergence of $(u,{\boldsymbol\sigma})$)}\label{NN1}
There exist a subsequence $(k',h')$ of $(k,h)$, with $k',h'\downarrow 0$, and a weak solution $(u,{\boldsymbol\sigma})$ of (\ref{modelf01}) in $(0,+\infty)$, such that $(\widetilde u_{h',k'}-m_0,\widetilde{\boldsymbol\sigma}_{h',k'})$ and $(u_{h',k'}-m_0,{\boldsymbol\sigma}_{h',k'})$  converge to $(u-m_0,{\boldsymbol\sigma})$ weakly-* in $L^\infty(0,+\infty;L^2(\Omega)\times \L^2(\Omega))$, weakly in $L^2(0,+\infty;H^1(\Omega)\times \H^1(\Omega))$ and strongly in $L^2(0,T;L^2(\Omega)\times \L^2(\Omega))$, for any $T>0$. 
\end{tma}
Note that, since the positivity of $u^n_h$ cannot be assured, then the positivity of the limit function $u$ cannot be proven in $3D$ domains. For $1D$ and $2D$ domains, the positivity of $u$ can be recovered a posteriori, using the existence and uniqueness  of (positive) weak solution $(u,{\boldsymbol\sigma})$ of (\ref{modelf01}), see \cite{FMD}.

On the other hand, by introducing the following functions:
\begin{itemize}
\item $\widetilde{v}_{h,k}$ are continuous functions on $[0,+\infty)$, linear on each interval $(t_{n-1},t_n)$ and equal to $v^n_h,$ at $t=t_n$, $n\geq 0$;
\item ${v}_{h,k}$ are the piecewise constant functions taking values $v^{n}_h$ on $(t_{n-1},t_n]$, $n\geq 1$,
\end{itemize}
proceeding as in  Lemma 4.12 of \cite{FMD} and taking into account the estimate (\ref{weak02UVlinvL2}), the following result can be proved:
\begin{lem} {\bf (Convergence of $v$)}\label{CN1}
There exist a subsequence $(k',h')$ of $(k,h)$, with $k',h'\downarrow 0$, and a weak solution $v$ of (\ref{modelf01eqv}) in $(0,+\infty)$, such that $\widetilde{v}_{h',k'}-m_0^2$ and $v_{h',k'}-m_0^2$ converge to $v-m_0^2$ weakly-* in $L^\infty(0,+\infty;L^2(\Omega))$, weakly in $L^2(0,+\infty;H^1(\Omega))$ and strongly in $L^2(0,T;L^2(\Omega))$, for any $T>0$.
\end{lem}

\begin{obs}
From the equivalence of problems (\ref{modelf00}) and (\ref{modelf01})-(\ref{modelf01eqv}) stablished in \cite{FMD}, and taking into account Theorem \ref{NN1} and Lemma  \ref{CN1}, we deduce that the limit pair $(u,v)$ is a weak-strong solution of problem (\ref{modelf00}).
\end{obs}

\subsection{Uniform strong estimates}
In this subsection, some a priori strong estimates  of the scheme \textbf{US} are obtained by assuming a regularity criterion (see (\ref{strong01}) below) which can be proved, at least, for 1D and 2D domains (see Theorem 4.22 of \cite{FMD}).
\begin{lem}\label{Lm:StrongIneq-n} {\bf(Strong inequality for $(u^n_h,{\boldsymbol\sigma}^n_h)$) }
	It holds 
	\begin{equation}\label{StrongIneq-n}
	\delta_t
	\Vert  (\hat u^n_h, {\boldsymbol\sigma}^n_h) \Vert_{1}^2 
	+ \Vert (\hat u^n_h, {\boldsymbol\sigma}^n_h)\Vert_{W^{1,6}}^2
	+  \Vert  (\delta_t \hat  u^n_h, \delta_t {\boldsymbol\sigma}^n_h) \Vert_{0}^2
	\le
	C_1 \Big(\Vert  (\hat u^n_h, {\boldsymbol\sigma}^n_h)  \Vert_{1}^2\Big)^d +
	C_2 \Vert  (\hat u^n_h, {\boldsymbol\sigma}^n_h) \Vert_{1}^2 
	\end{equation}
	where $d=2$ for $2D$ domains and $d=3$ for $3D$ domains.
\end{lem}
\begin{proof} The proof follows as in Lemma 4.14 of \cite{FMD}, but in this case it is necessary to use the estimates (\ref{pd0a})-(\ref{pd0b}).
\end{proof}

\begin{cor} \label{stlem}{\bf(Strong estimates for $(u^n_h,{\boldsymbol \sigma}^n_h)$)} 
Let $(u_0, v_0)\in H^1(\Omega)\times H^2(\Omega)$ and $(u^n_h,{\boldsymbol \sigma}^n_h)$ be a solution of the scheme \textbf{US}. Assuming 
the following regularity criterion:
\begin{equation}\label{strong01}
\Vert (u^n_h, {\boldsymbol
	\sigma}^n_h)\Vert_{1}^{2}\leq K_0, \ \ \ \forall n\geq 0,
\end{equation}
then the following  estimate holds
\begin{equation}\label{strong02}
	 k \underset{m=1}{\overset{n}{\sum}}(\Vert (\delta_t {u}^m_h, \delta_t {\boldsymbol \sigma}^m_h)\Vert_{0}^2 + \Vert (\hat{u}^m_h,{\boldsymbol \sigma}^m_h)\Vert_{W^{1,6}}^2)\leq K_1,  \ \ \ \forall n\geq 1,
\end{equation}
\end{cor}
\begin{proof}
The proof follows by using (\ref{weak01}) and (\ref{strong01}) in (\ref{StrongIneq-n}).
\end{proof}

\begin{cor} \label{mslem} {\bf(Regular estimates for $(u^n_h,{\boldsymbol \sigma}^n_h)$)}
Assume that $(u_0,{\boldsymbol\sigma}_0)\in H^2(\Omega)\times \H^2(\Omega)$. Under the hypothesis of Corollary \ref{stlem}, the fo\-llo\-wing estimates hold
\begin{equation}\label{stdta}
\Vert (\delta_t u^n_h, \delta_t {\boldsymbol \sigma}^n_h)\Vert^2_{0} + k \underset{m= 1}{\overset{n}{\sum}}\Vert (\delta_t u^m_h,
\delta_t {\boldsymbol \sigma}^m_h) \Vert^{2}_{1} \leq K_2, \ \ \forall n\geq 1,
\end{equation}
\begin{equation}\label{stdta2}
\Vert (u^n_h, {\boldsymbol \sigma}^n_h)\Vert^2_{W^{1,6}} \leq K_3, \ \ \forall n\geq 0,
\end{equation}
\end{cor}
\begin{proof}
The proof follows as in Corollary 4.18 of \cite{FMD}.
\end{proof}
\begin{obs}
In particular, from (\ref{stdta2}) one has $
\Vert (u^n_h,{\boldsymbol \sigma}^n_h)\Vert_{L^\infty} \leq K_4$ for all $n\geq 0$.
\end{obs}

\begin{lem}\label{lemVs}{\bf (Strong estimates for $v^n_h$)}
Let $v^n_h$ be the solution of (\ref{edovf}). Under hypotheses of Corollary \ref{stlem}, the following estimate holds
\begin{equation}\label{strong01V}
\Vert v^n_h\Vert_{1}^{2} + k \underset{m=1}{\overset{n}{\sum}}(\Vert \delta_t \hat{v}^m_h\Vert_{0}^2  + \Vert {A}_h \hat{v}^m_h\Vert_{0}^2)\leq C_1, \ \ \ \forall n\geq 1.
\end{equation}
\end{lem}
\begin{proof}
Taking $\bar{v}={A}_h \hat{v}^n_h$ and $\delta_t \hat{v}^n_h$ in (\ref{edovf-nnn}), one has
\begin{equation}\label{vn1V}
\delta_t \left( \Vert  \hat{v}^n_h \Vert_1^2 \right) + \frac{1}{2}\Vert {A}_h \hat{v}^n_h \Vert_0^2 +\frac{1}{2}\Vert \delta_t \hat{v}^n_h \Vert_0^2
   \leq  C \Vert \hat{u}^n_h+2m_0\Vert_{L^4}^2 \Vert \hat{u}^n_h\Vert_{L^4}^2.
\end{equation}
Then, multiplying (\ref{vn1V}) by $k$, adding for $m=0,\cdot \cdot \cdot,n$,  and using (\ref{weak01}) and (\ref{strong01}), (\ref{strong01V})  is deduced.
\end{proof}

\begin{tma} {\bf (Regular estimates for $v^n_h$)}
Assume  $v_0\in H^2(\Omega)$. Under the hypotheses of Corollary \ref{mslem}, the fo\-llo\-wing estimates hold
\begin{equation}\label{stdtaV}
\Vert \delta_t v^n_h\Vert^2_{0}+ k \underset{m= 1}{\overset{n}{\sum}}\Vert \delta_t \hat{v}^m_h \Vert^{2}_{1} \leq C_2, \ \ \forall n\geq 1,
\end{equation}
\begin{equation}\label{stdtaV2}
\Vert v^n_h\Vert^2_{W^{1,6}} \leq C_3, \ \ \forall n\geq 0.
\end{equation}
\end{tma}
\begin{proof}
We denote $\widetilde{v}^n_h:=\delta_t \hat{v}^n_h$. Then, making the time discrete derivative of (\ref{edovf-nnn}) (using that $\delta_t({u}^n_h)^2=({u}^n_h + {u}^{n-1}_h)\delta_t {u}^n_h$),
testing by $\widetilde{v}^n_h$ and using (\ref{strong01}), one has
\begin{equation}\label{vn3V}
\frac{1}{2}\delta_t \left( \Vert  \widetilde{v}^n_h \Vert_0^2 \right) + \frac{1}{2}\Vert  \widetilde{v}^n_h \Vert_1^2    \leq C \Vert {u}^n_h + {u}^{n-1}_h \Vert_{L^3}^2 \Vert \delta_t {u}^n_h\Vert_{0}^2\leq C\Vert \delta_t {u}^n_h\Vert_{0}^2.
\end{equation}
Then, multiplying (\ref{vn3V}) by $k$, adding for $m=2,\cdot \cdot \cdot, n$ and using (\ref{strong02}), one arrives at
\begin{equation*}
\Vert  \widetilde{v}^n_h \Vert_0^2 + k \underset{m= 1}{\overset{n}{\sum}}\Vert \widetilde{v}^m_h \Vert^{2}_{1}   \leq C + C \Vert  \widetilde{v}^1_h \Vert_0^2.
\end{equation*}
Then, in order to deduce \eqref{stdtaV}, it suffices to bound $\Vert  \widetilde{v}^1_h \Vert_0^2$. Indeed, from (\ref{edovf-nnn}), one has
\begin{equation}\label{edt00}
(\delta_t \hat{v}^1_h,\bar{v}_h)  +({A}_h (\hat{v}^1_h - \hat{v}^0_h),\bar{v}_h) +({A}_h \hat{v}^0_h,\bar{v}_h)  =((\hat{u}^1_h + 2m_0) \hat{u}^1_h,\bar{v}_h), \ \ \forall
\bar{v}_h\in V_h.
\end{equation}
Then, taking $\bar{v}_h= \delta_t \hat{v}^1_h$ in (\ref{edt00}) and using (\ref{strong01}), one can obtain
\begin{equation}\label{edt01}
\Vert \delta_t \hat{v}^1_h \Vert^2_0 \leq  C\Vert {A}_h \hat{v}^0_h\Vert_0^2 + C\Vert \hat{u}^1_h\Vert_{L^4}^2\Vert \hat{u}^1_h+2m_0\Vert_{L^4}^2.
\end{equation}
From the inverse inequality \eqref{16inv1} and the interpolation error (\ref{EI03}), we have
\begin{eqnarray}\label{v0a}
\Vert {A}_h \hat{v}^0_h\Vert_0
 \leq \Vert {A}_h (\mathcal{R}^v_h \hat{v}_0 - \hat{v}_0)\Vert_0 +  \Vert {A}_h \hat{v}_0\Vert_0
 \leq C\frac{1}{h}\Vert \mathcal{R}^v_h \hat{v}_0 - \hat{v}_0 \Vert_1 
+ \Vert \hat{v}_0\Vert_2\leq C \Vert \hat{v}_0\Vert_2.
\end{eqnarray}
Thus, using (\ref{strong01}) and (\ref{v0a}) in (\ref{edt01}), the estimate $\Vert  \widetilde{v}^1_h \Vert_0^2\leq C$ is obtained. Finally, (\ref{stdtaV2}) can be deduced from  (\ref{pd0a})$_2$, (\ref{strong01}) and (\ref{stdtaV}).
\end{proof}

\subsection{Error estimates}
We will obtain error estimates 
for 
the scheme \textbf{US} with respect to a su\-ffi\-cien\-tly regular solution $(u,{\boldsymbol\sigma})$ of (\ref{modelf01}) and  $v$ of (\ref{modelf01eqv}). 
 For any final time $T>0$, let us consider a fixed partition of $[0,T]$ given by $(t_n=nk)_{n=0}^N$, where $k=T/N>0$ is the time step. We will denote by   $C,C_i,K_i$ to different positive constants possibly  depending on the continuous solution $(u, v,{\boldsymbol \sigma}=\nabla v)$, but independent of the discrete parameters $(k,h)$ and the length of the time interval $T$, because the dependence of $T$ will be given explicitly. In order to obtain optimal error estimates, we will  assume the following continuous FE  spaces: 
$$U_h, {\boldsymbol\Sigma}_h\sim\mathbb{P}_m[\x]
\quad \hbox{and}\quad V_h\sim\mathbb{P}_{m+1},\quad \hbox{with $m\geq 1$.}
$$
 This is a natural assumption because, in the continuous model, the energy norm for $v$ has one order higher than for $(u,{\boldsymbol\sigma})$. In fact, we are going to obtain optimal error estimates, in weak norms for $(u,\boldsymbol \sigma)$ and in strong norms for $v$.
 
 \
 
 We  introduce the following notations for the errors at $t=t_{n}$:  
$$
e_u^n=u(t_n)-u^n_h,\quad e_{\boldsymbol\sigma}^n={\boldsymbol\sigma}(t_n)-{\boldsymbol\sigma}^n_h
\quad \hbox{and} \quad e_v^n=v(t_n)-v^n_h
  $$ 
and for the discrete norms:
$$
\Vert  (e^n)\Vert^2_{l^\infty X} := \max_{n=1,\cdots,N} \| e^n\|^2_X,
\quad 
\Vert  (e^n)\Vert^2_{l^2 X} := k \sum_{n=1}^N \| e^n\|^2_X.
$$
\subsubsection{Error estimates for $(e_u^n,e_{{\boldsymbol \sigma}}^n)$ in weak norms}

Subtracting (\ref{modelf01}) at $t=t_n$ and the scheme \textbf{US}, then $(e_u^n,e_{{\boldsymbol \sigma}}^n)$ satisfies
\begin{equation}\label{erru}
\left(\delta_t e_u^n,\bar{u}_h\right) + (\nabla e_u^n, \nabla \bar{u}_h) +(e_u^n{\boldsymbol \sigma}(t_n)+u^n_h e_{\boldsymbol \sigma}^n,\nabla \bar{u}_h)=(\xi_1^n,\bar{u}_h), \ \ \forall \bar{u}_h\in U_h,
\end{equation}
\begin{equation}\label{errs}
\left(\delta_t e_{\boldsymbol\sigma}^n,\bar{\boldsymbol \sigma}_h\right)+ \langle 
B_h e_{\boldsymbol\sigma}^n, \bar{\boldsymbol \sigma}_h\rangle = 2(e_u^n\nabla u(t_n)+u^n_h \nabla e_u^n ,\bar{\boldsymbol \sigma}_h) + (\xi_2^n,\bar{\boldsymbol \sigma}_h),\ \ \forall \bar{\boldsymbol
\sigma}_h\in {\boldsymbol \Sigma}_h,
\end{equation}
where $\xi_1^n,\xi_2^n$ are the consistency errors associated to the scheme \textbf{US}, that is, 
$$
\xi_1^n=\delta_t (u(t_n)) - u_t(t_n)
\quad \hbox{and}\quad
\xi_2^n=\delta_t ({\boldsymbol\sigma}(t_n)) - {\boldsymbol\sigma}_t(t_n).
$$
 Now, considering the interpolation operators $\mathcal{R}_h^u$ and $\mathcal{R}_h^{\boldsymbol \sigma}$ defined in (\ref{interp2})-(\ref{interp1}), the errors $e_u^n$ and $e_{\boldsymbol\sigma}^n$ are decomposed as follows
\begin{equation}\label{u1a}
e_u^n=(\mathcal{I}-\mathcal{R}_h^u)u(t_n) + \mathcal{R}_h^u u(t_n) - u^n_h=e_{u,i}^n+e_{u,h}^n,
\end{equation}
\begin{equation}\label{s1a}
e_{\boldsymbol\sigma}^n=(\mathcal{I}-\mathcal{R}_h^{\boldsymbol\sigma}){\boldsymbol\sigma}(t_n) + \mathcal{R}_h^{\boldsymbol\sigma} {\boldsymbol\sigma}(t_n) - {\boldsymbol\sigma}^n_h=e_{{\boldsymbol\sigma},i}^n+e_{{\boldsymbol\sigma},h}^n,
\end{equation}
where $e_{u,i}^n$ is the interpolation error and $e_{u,h}^n$ is the discrete error of $u$ (idem for $\boldsymbol\sigma$).
Then, taking into account (\ref{interp2})-(\ref{interp1}), from (\ref{erru})-(\ref{s1a}),  one has
\begin{eqnarray}\label{int003}
&\left(\delta_t e_{u,h}^n,\bar{u}_h\right)&\!\!\!\!\! + (\nabla e_{u,h}^n, \nabla \bar{u}_h) +(e_{u,h}^n{\boldsymbol \sigma}(t_n)+u^n_h e_{{\boldsymbol \sigma},h}^n,\nabla \bar{u}_h)=(\xi_1^n,\bar{u}_h)\nonumber\\
&&\!\!\!\!\! -\left(\delta_t e_{u,i}^n,\bar{u}_h\right) -(e_{u,i}^n{\boldsymbol \sigma}(t_n)+u^n_h e_{{\boldsymbol \sigma},i}^n,\nabla \bar{u}_h), \ \ \forall \bar{u}_h\in U_h,
\end{eqnarray}
\vspace{-1cm}
\begin{eqnarray}\label{int004}
&\left(\delta_t e_{{\boldsymbol\sigma},h}^n,\bar{\boldsymbol \sigma}_h\right) &\!\!\!\!\!
+ ( B_h
e_{{\boldsymbol\sigma},h}^n, \bar{\boldsymbol \sigma}_h)
= (\xi_2^n,\bar{\boldsymbol \sigma}_h)+2(e_{u,h}^n\nabla u(t_n)+u^n_h \nabla e_{u,h}^n ,\bar{\boldsymbol \sigma}_h) \nonumber\\
&&\!\!\!\!\!+2(e_{u,i}^n\nabla u(t_n)+u^n_h \nabla e_{u,i}^n ,\bar{\boldsymbol \sigma}_h) -\left(\delta_t e_{{\boldsymbol\sigma},i}^n,\bar{\boldsymbol \sigma}_h\right),\ \ \forall \bar{\boldsymbol
\sigma}_h\in {\boldsymbol\Sigma}_h.
\end{eqnarray}
Notice that $\int_\Omega e^n_{u,h}=0$ (since $u^0_h= \mathcal{R}_h^u u_0$ and from (\ref{interp2}) $\int_\Omega \mathcal{R}_h^u u(t_n)=\int_\Omega u(t_n) = m_0$), hence the following norms are equivalent: $\Vert \nabla e_{u,h}^n\Vert _0\simeq \Vert e_{u,h}^n\Vert_1$.

\begin{tma}\label{erteo}
Assume that there exists  $(u,{\boldsymbol\sigma})$ an exact solution of (\ref{modelf01}) such that:
\begin{equation}\label{regU}
\left\{
\begin{array}
[c]{ccc}
(u,{\boldsymbol\sigma})\in L^{\infty}(0,+\infty;H^{m+1}(\Omega)\!\times\! \H^{m+1}(\Omega)), \ (u_{t},{\boldsymbol\sigma}_{t})\in
L^{2}(0,+\infty;H^{m+1}(\Omega)\!\times\! \H^{m+1}(\Omega)), \\
(u_{tt},{\boldsymbol\sigma}_{tt})\in L^2(0,+\infty;H^1(\Omega)'\times \H^1_\sigma(\Omega)').
\end{array}
\right.
\end{equation}
Let $(u^n_h,{\boldsymbol\sigma}^n_h)$ be a solution of the scheme \textbf{US}.  Then, if 
\begin{equation}\label{re01}
k (\Vert (u, {\boldsymbol\sigma})\Vert_{L^\infty(H^1)}^4+\Vert(u,{\boldsymbol\sigma})\Vert_{L^\infty(H^2)}^2) \quad \hbox{is small enough}, 
\end{equation}
 the following a priori error estimate holds
\begin{eqnarray}\label{priorierr}
\Vert  (e_{u,h}^n,e_{{\boldsymbol \sigma},h}^n)\Vert_{l^\infty L^2\cap l^2 H^1}^2
\leq 
K_1 T \exp(K_2T)(k^2+h^{2(m+1)}).
\end{eqnarray}
\end{tma}
Recall that $u$ and ${\boldsymbol\sigma}$ are approximated by $\mathbb{P}_{m}$-continuous FE.
\begin{proof}
Taking $\bar{u}_h=e_{u,h}^n$ in (\ref{int003}), $\bar{\boldsymbol
\sigma}_h=\displaystyle\frac{1}{2}e_{{\boldsymbol\sigma},h}^n$ in (\ref{int004}) and adding, the terms $(u^n_h \nabla e_{u,h}^n ,e_{{\boldsymbol \sigma},h}^n)$ cancel, and we obtain
\begin{eqnarray}\label{ester1}
&\delta_t&\!\!\!\!\!\left( \displaystyle \frac{1}{2} \Vert e_{u,h}^n\Vert_{0}^2 + \displaystyle \frac{1}{4} \Vert e_{{\boldsymbol \sigma},h}^n\Vert_{0}^2 \right)+\frac{1}{2}\Vert  (e_{u,h}^n,e_{{\boldsymbol \sigma},h}^n)\Vert_{1}^2= (\xi_1^n,e_{u,h}^n) +\frac{1}{2}(\xi_2^n,e_{{\boldsymbol \sigma},h}^n) -\left(\delta_t e_{u,i}^n,e_{u,h}^n\right)  \nonumber\\
&&\!\!\!\!\!  -\frac{1}{2}\left(\delta_t e_{{\boldsymbol\sigma},i}^n,e_{{\boldsymbol \sigma},h}^n\right) -(e_{u,h}^n\, , \, {\boldsymbol \sigma}(t_n)\cdot \nabla e_{u,h}^n - \nabla u(t_n)\cdot e_{{\boldsymbol \sigma},h}^n) 
-(e_{u,i}^n \, , \, {\boldsymbol \sigma}(t_n)\cdot\nabla e_{u,h}^n 
- \nabla u(t_n)\cdot e_{{\boldsymbol \sigma},h}^n)
\nonumber\\
&&\!\!\!\!\! -(u^n_h\, , \, e_{{\boldsymbol \sigma},i}^n\cdot \nabla e_{u,h}^n - \nabla e_{u,i}^n \cdot e_{{\boldsymbol \sigma},h}^n)  := \underset{m=1}{\overset{7}{\sum}} I_m.
\end{eqnarray}
Then, using the H\"older and Young inequalities, the 3D interpolation inequality (\ref{in3D}), the interpolation errors (\ref{EI01})-(\ref{EI02}), the stability property (\ref{SP01}) and the hypothesis (\ref{regU}),  the terms on the right hand side of (\ref{ester1}) can be estimated as follows 
\begin{eqnarray}\label{Ea1a}
&I_1+ I_2&\!\!\!\leq \varepsilon \Vert (e_{u,h}^n, e_{{\boldsymbol \sigma},h}^n)\Vert_{1}^2 + C_{\varepsilon}\Vert (\xi_1^n,\xi_2^n)\Vert_{(H^1)' \times(H^1_{\sigma})'}^2 \nonumber\\
&&\!\!\!\leq \varepsilon\Vert (e_{u,h}^n, e_{{\boldsymbol \sigma},h}^n)\Vert_{1}^2+C k \int_{t_{n-1}}^{t_n}\Vert (u_{tt}(t),{\boldsymbol \sigma}_{tt}(t))\Vert_{(H^1)'\times(H_{\sigma}^1)'}^2dt,
\end{eqnarray}
\begin{eqnarray}\label{Ea1b}
&I_3+ I_4&\!\!\!\leq  \Vert (e_{u,h}^n, e_{{\boldsymbol \sigma},h}^n)\Vert_{0}\Vert((\mathcal{I} - \mathcal{R}_h^u) \delta_t u(t_n), (\mathcal{I} - \mathcal{R}_h^{\boldsymbol \sigma}) \delta_t {\boldsymbol \sigma}(t_n))\Vert_{0} \nonumber\\
&&\!\!\!\leq \varepsilon\Vert (e_{u,h}^n, e_{{\boldsymbol \sigma},h}^n)\Vert_{1}^2+C h^{2(m+1)}\Vert (\delta_t u(t_n), \delta_t {\boldsymbol \sigma}(t_n))\Vert_{m+1}^2\nonumber\\
&&\!\!\!\leq \varepsilon\Vert (e_{u,h}^n, e_{{\boldsymbol \sigma},h}^n)\Vert_{1}^2+\displaystyle\frac{C h^{2(m+1)}}{k}\int_{t_{n-1}}^{t_n}\Vert(u_t, {\boldsymbol \sigma}_t ) \Vert_{m+1}^2  dt,
\end{eqnarray}
where the fact that $(\delta_t u(t_n),\delta_t {\boldsymbol \sigma}(t_n))=\displaystyle\frac{1}{k}\int_{t_{n-1}}^{t_n}(u_t, {\boldsymbol \sigma}_t )$ was used in the last inequality,
\begin{eqnarray}\label{Ea1}
I_5\leq \Vert e_{u,h}^n\Vert_{L^3} \Big(\Vert \nabla u(t_n)\Vert_{0} \Vert e_{{\boldsymbol \sigma},h}^n\Vert_{L^6} + \Vert \nabla \cdot {\boldsymbol\sigma}(t_n)\Vert_{0} \Vert e_{u,h}^n \Vert_{L^6} \Big)\nonumber\\
\leq \varepsilon\Vert (e_{u,h}^n, e_{{\boldsymbol \sigma},h}^n)\Vert_{1}^2+C_{\varepsilon}
\Vert (\nabla u, \nabla \cdot {\boldsymbol\sigma})\Vert_{L^\infty(L^2)}^4
\Vert e_{u,h}^n\Vert_{0}^2,
\end{eqnarray}
\begin{eqnarray}\label{Ea2}
&I_6&\!\!\!\leq \Vert e_{u,i}^n\Vert_{0} \Big(\Vert \nabla e_{u,h}^n\Vert_{0} \Vert {\boldsymbol \sigma}(t_n)\Vert_{L^\infty} + \Vert \nabla u(t_n)\Vert_{L^3} \Vert e_{{\boldsymbol \sigma},h}^n \Vert_{L^6} \Big)\nonumber\\
&&\!\!\!\leq \varepsilon\Vert  (e_{u,h}^n,e_{{\boldsymbol \sigma},h}^n)\Vert_{1}^2
+C_{\varepsilon}
\Vert e_{u,i}^n\Vert_{0}^2\nonumber
\leq \varepsilon\Vert  (e_{u,h}^n,e_{{\boldsymbol \sigma},h}^n)\Vert_{1}^2
+C \, h^{2(m+1)},
\end{eqnarray}
\vspace{-1,3 cm}

\begin{eqnarray}\label{Ea3}
&I_7&\!\!\!\leq \vert(e_{u,h}^n\, ,\, e_{{\boldsymbol \sigma},i}^n\cdot \nabla e_{u,h}^n - \nabla e_{u,i}^n \cdot e_{{\boldsymbol \sigma},h}^n)\vert + \vert(\mathcal{R}_h^u u(t_n)\,,\, e_{{\boldsymbol \sigma},i}^n\cdot \nabla e_{u,h}^n - \nabla e_{u,i}^n \cdot e_{{\boldsymbol \sigma},h}^n)\vert  \nonumber\\
&&\!\!\!\leq \varepsilon\Vert (e_{u,h}^n, e_{{\boldsymbol \sigma},h}^n)\Vert_{1}^2+C_{\varepsilon}\Vert e_{u,h}^n\Vert_{0}^2\Vert (e_{u,i}^n,e_{{\boldsymbol \sigma},i}^n)\Vert_{W^{1,3}\times L^\infty}^2+C_{\varepsilon}\Vert \mathcal{R}_h^u u(t_n)\Vert_{W^{1,3}\cap L^\infty}^2\Vert (e_{u,i}^n,e_{{\boldsymbol \sigma},i}^n)\Vert_{0}^2\nonumber\\
&&
\!\!\!\leq \varepsilon\Vert (e_{u,h}^n, e_{{\boldsymbol \sigma},h}^n)\Vert_{1}^2\!+\! C
\Vert(u,{\boldsymbol\sigma})\Vert_{L^\infty(H^2)}^2
\Vert e_{u,h}^n \Vert_{0}^2\!+\!Ch^{2(m+1)}.
\end{eqnarray}
 Therefore, taking $\varepsilon$ small enough, from (\ref{ester1})-(\ref{Ea3}) we obtain
\begin{eqnarray}\label{ee001}
&\delta_t&\!\!\!\!\!\left( \displaystyle \frac{1}{2} \Vert e_{u,h}^n\Vert_{0}^2+ \frac{1}{4} \Vert e_{{\boldsymbol \sigma},h}^n\Vert_{0}^2\right)+\Vert (e_{u,h}^n,e_{{\boldsymbol \sigma},h}^n)\Vert_{1}^2
\le Ch^{2(m+1)} 
+
 C(u, {\boldsymbol\sigma}) \Vert e_{u,h}^n\Vert_{0}^2
\nonumber\\
&&\!\!\!\!\! + C k \int_{t_{n-1}}^{t_n}\Vert (u_{tt}(t),{\boldsymbol \sigma}_{tt}(t))\Vert_{(H^1)'\times(H_{\sigma}^1)'}^2dt
+\displaystyle\frac{C h^{2(m+1)}}{k}\int_{t_{n-1}}^{t_n}\Vert(u_t, {\boldsymbol \sigma}_t ) \Vert_{m+1}^2  dt
\end{eqnarray}
 where $C(u, {\boldsymbol\sigma})= C\Big(\Vert (u, {\boldsymbol\sigma})\Vert_{L^\infty(0,\infty;H^1)}^4+\Vert(u,{\boldsymbol\sigma})\Vert_{L^\infty(0,\infty;H^2)}^2\Big)$.
Then, multiplying (\ref{ee001}) by $k$, adding from $n=1$ to $n=r$, recalling that $e_{u,h}^0=e_{{\boldsymbol \sigma},h}^0=0$, and taking into account (\ref{regU}), 
 it holds
\begin{equation*}
\Big[\frac{1}{4}- k\, C(u, {\boldsymbol\sigma})\Big] \Vert (e_{u,h}^r,e_{{\boldsymbol \sigma},h}^r)\Vert_{0}^2+k\sum_{n=1}^{r}\Vert (e_{u,h}^n,e_{{\boldsymbol \sigma},h}^n)\Vert_{1}^2\leq C k^2 +C h^{2(m+1)} 
+C\, k \sum_{n=0}^{r-1} \Vert e_{u,h}^n\Vert_{0}^2.
\end{equation*}
Therefore, assuming the hypothesis (\ref{re01}) and using the  discrete Gronwall Lemma, 
 error estimate (\ref{priorierr}) can be deduced. 
\end{proof}
\begin{obs}\label{remH2}
Under the hypotheses of Theorem \ref{erteo}, one has in particular
$$
\Vert (u_h^n,{\boldsymbol\sigma}_h^n)\Vert_1^2 \leq C +
 K_1 T \exp(K_2T)\Big(k+\frac{h^{2(m+1)}}{k}\Big).
$$
Therefore, 
under the hypothesis 
\begin{equation}\label{uub}
\frac{h^{2(m+1)}}{k}\leq C,
\end{equation}
one has the estimate
\begin{equation}\label{uuc}
\Vert (u_h^n,{\boldsymbol\sigma}_h^n)\Vert_1^2 \leq C ,
\end{equation}
hence the hypothesis (\ref{uniq01}) providing uniqueness of the scheme  is reduced to $k$ small enough.
Finally, since  for any choice of $(k,h)$ either (\ref{uua}) (see Remark \ref{inddom01}) or (\ref{uub}) hold, one has the uniqueness of $(u^n_h,{\boldsymbol\sigma}^n_h)$ solution of (\ref{modelf02}) only imposing $k$ small enough.
\end{obs}

\subsubsection{Error estimates for $e_v^n$ is strong norms}
Subtracting (\ref{modelf01eqv}) at $t=t_n$ and (\ref{edovf}), then $e_v^n$ satisfies
\begin{equation} \label{errv}
(\delta_t e_v^n,\bar{v}_h) + \langle A
e_v^n,\bar{v}_h\rangle = ((u(t_n)+u^n_h)e_u^n,\bar{v}_h)+ (\xi_3^n,\bar{v}_h),\ \ \forall \bar{v}_h\in V_h,
\end{equation}
where $\xi_3^n=\delta_t (v(t_n)) - v_t(t_n)$ is the consistency error associated to (\ref{edovf}). Now, considering the interpolation operator $\mathcal{R}_h^v$ defined in (\ref{interp2a}),  $e_v^n$ is decomposed as follows
\begin{equation}\label{v1a}
e_v^n=(\mathcal{I}-\mathcal{R}_h^v)v(t_n) + \mathcal{R}_h^v v(t_n) - v^n_h=e_{v,i}^n+e_{v,h}^n .
\end{equation}
Then, taking into account (\ref{interp2a}), from (\ref{errv})-(\ref{v1a}), one has for all $\bar{v}_h\in V_h$:
\begin{eqnarray} \label{int01a}
&\left(\delta_t e_{v,h}^n,\bar{v}_h\right)&\!\!\!\!\! + ({A}_h
e_{v,h}^n,\bar{v}_h) =  (\xi_3^n,\bar{v}_h)+((u(t_n)+u^n_h)(e_{u,h}^n\!+e_{u,i}^n),\bar{v}_h)\!-\!(\delta_t e_{v,i}^n,\bar{v}_h).
\end{eqnarray}

\begin{tma}[\bf Strong estimates] \label{erteoV}
Under the hypotheses of Theorem \ref{erteo}, and assuming the regularity:
\begin{equation}\label{regV}
(v_t,v_{tt})\in
L^{2}(0,+\infty;H^{m+1}(\Omega)\times L^2(\Omega)),
\end{equation}
the following a priori error estimate holds
\begin{equation}\label{priorierr001}
\Vert  e_{v,h}^n\Vert_{l^\infty H^1\cap l^2 W^{1,6}}^2
\leq 
 K_3T \exp(K_4T)(k^2+h^{2(m+1)}).
\end{equation}
\end{tma}
\begin{proof}
Taking $\bar{v}_h={A}_h e_{v,h}^n$ in (\ref{int01a}) and using the H\"older and Young inequalities,  one has 
\begin{eqnarray}\label{ester3v}
&\delta_t&\!\!\!\!\!\left( \displaystyle \frac{1}{2} \Vert e_{v,h}^n\Vert_{1}^2 \right)+\frac{k}{2}\Vert \delta_t e_{v,h}^n\Vert_1^2+\frac{1}{2}\Vert  {A}_h e_{v,h}^n\Vert_{0}^2  \leq C\Vert \xi_3^n\Vert_{0}^2+C \Vert  u(t_n)+u^n_h\Vert_{L^3}^2\Vert e_{u,h}^n\Vert_{L^6}^2 \nonumber\\
&& + C\Vert (u(t_n)+u^n_h) e_{u,i}^n\Vert_{0}^2+C\Vert (\mathcal{I} - \mathcal{R}_h^v) \delta_t v(t_n)\Vert_{0}^2 .
\end{eqnarray}
Using the H\"older inequality, the interpolation error (\ref{EI01}), the stability property (\ref{SP01}) and the hypothesis (\ref{regU}),  one has
\begin{eqnarray}\label{errr0a}
&\Vert (u(t_n)+u^n_h) e_{u,i}^n\Vert_{0}^2&\!\!\!\leq C\Vert u(t_n) + \mathcal{R}_h^u u(t_n)\Vert_{L^\infty}^2\Vert e_{u,i}^n\Vert_{0}^2 + C\Vert e_{u,h}^n\Vert_{L^6}^2 \Vert e_{u,i}^n\Vert_{L^3}^2\nonumber\\
&&\!\!\! \leq C h^{2(m+1)} + C\Vert e_{u,h}^n\Vert_{1}^2.
\end{eqnarray}
Therefore, proceeding as in (\ref{Ea1a}) and (\ref{Ea1b}) and using (\ref{errr0a}), then 
 (\ref{ester3v}) becomes 
\begin{eqnarray*}
&\delta_t \left( \Vert e_{v,h}^n\Vert_{1}^2 \right)&\!\!\!\!\!+\Vert  {A}_h^v e_{v,h}^n\Vert_{0}^2  \leq C k \int_{t_{n-1}}^{t_n}\Vert v_{tt}(t)\Vert_{0}^2dt
+Ch^{2(m+1)}
\nonumber\\
&&\!\!\!\!\!\!\!\!\!+(C \Vert  u(t_n) +u^n_h\Vert_{L^3}^2+C)\Vert e_{u,h}^n\Vert_{1}^2
+\displaystyle\frac{C h^{2(m+1)}}{k}\int_{t_{n-1}}^{t_n}\Vert v_t \Vert_{m+1}^2 dt.
\end{eqnarray*}
Now, in order to bound the term $\Vert  u(t_n)+u^n_h\Vert_{L^3}^2$, we split the argument into two cases:
\begin{enumerate}
\item{{\bf Estimates assuming $h<< f(k)$ ($h$ small enough with respect to $k$):}\\
From (\ref{priorierr}) one has that $k\underset{n=1}{\overset{r}{\sum}} \Vert e_{u,h}^n\Vert_{1}^2\leq 
 K_1T \exp(K_2T) (k^2 +
h^{2(m+1)})$, which implies  
\begin{equation}\label{estInF}
\Vert e_{u,h}^n\Vert_{1}\leq 
 K_1T^{1/2} \exp(K_2T) \Big(k^{1/2} + \frac{h^{m+1}}{k^{1/2}} \Big).
\end{equation}
Moreover, using  (\ref{in3D}),  (\ref{SP01}), (\ref{regU}), (\ref{priorierr}) and (\ref{estInF}), 
 one obtains
\begin{eqnarray*}\label{DIE12}
&\Vert u(t_n) +u^n_h\Vert_{L^3}^2&\!\!\!\!\leq C \Vert u(t_n)\Vert_{L^3}^2 + C \Vert \mathcal{R}_h^u u(t_n)\Vert_{L^3}^2 + C\Vert e_{u,h}^n\Vert_{L^3}^2 \leq C+ C\Vert e_{u,h}^n\Vert_{0}\Vert e_{u,h}^n\Vert_{1}\nonumber\\
&& \!\!\!\! \leq C + 
 K_1T \exp(K_2T) (k+h^{m+1}) \Big(k^{1/2} + \frac{h^{m+1}}{k^{1/2}}\Big) ,
\end{eqnarray*}
hence $\Vert u(t_n) +u^n_h\Vert_{L^3}^2\leq C$ 
assuming the hypothesis 
\begin{equation}\label{k1}
\frac{h^{2(m+1)}}{k^{1/2}}\leq \frac{C}{
 K_1T \exp(K_2T)}.
\end{equation}}
\item{{\bf Estimates assuming $k<< g(k)$ ($k$ small enough with respect to $h$):}\\
Using the inverse inequality $\Vert u_h\Vert_{L^{3}}\leq \frac{C}{h^{1/2}}\Vert u_h\Vert_{0}$ for all $u_h\in U_h$, (\ref{SP01}), (\ref{regU}) and (\ref{priorierr}), 
\begin{eqnarray*}
&\Vert u(t_n) +u^n_h\Vert_{L^3}^2&\!\!\!\!\leq C \Vert u(t_n)\Vert_{L^3}^2 + C \Vert \mathcal{R}_h^u u(t_n)\Vert_{L^3}^2 + C\Vert e_{u,h}^n\Vert_{L^3}^2 \nonumber\\
&&\!\!\!\!  \leq  C+\frac{C}{h}\Vert e_{u,h}^n\Vert_{0}^2 \leq   C+
 K_1T \exp(K_2T) \frac{1}{h} (k^2 +h^{2(m+1)}) ,
\end{eqnarray*}
hence $\Vert u(t_n) +u^n_h\Vert_{L^3}^2\leq C$ 
assuming the hypothesis
\begin{equation}\label{k2}
\frac{k^2}{h}\le \frac{C}{
 K_1T \exp(K_2T)}.
\end{equation}}
\end{enumerate}
Therefore, since for any choice of $(k,h)$ either (\ref{k1}) or (\ref{k2}) hold, 
 one always obtains
\begin{eqnarray}\label{new02b}
&\delta_t \left( \Vert e_{v,h}^n\Vert_{1}^2 \right)&\!\!\!\!\!+\Vert  {A}_h e_{v,h}^n\Vert_{0}^2  \leq C k \int_{t_{n-1}}^{t_n}\Vert v_{tt}(t)\Vert_{0}^2dt\nonumber\\
&&\!\!\!\!\!\!\!\!\!+C\Vert e_{u,h}^n\Vert_{1}^2+Ch^{2(m+1)}+\displaystyle\frac{C h^{2(m+2)}}{k}\int_{t_{n-1}}^{t_n}\Vert v_t \Vert_{m+2}^2 dt.
\end{eqnarray}
Multiplying (\ref{new02b}) by $k$, adding from $n=1$ to $n=r$, recalling that $e_{v,h}^0=0$ and using  (\ref{priorierr}) and (\ref{regV}), 
the error estimate (\ref{priorierr001}) can be obtained. 
\end{proof}

\section{Linear iterative methods to approach the scheme \textbf{US}}
 Since the nonlinear scheme \textbf{US}   cannot be directly implemented, we propose two linear iterative methods to approach a solution $(u^n_h,{\boldsymbol\sigma}^n_h)$ of the scheme \textbf{US};  a Picard method and  Newton's method. 
The solvability of both methods and the convergence towards \textbf{US} will be proved.
\subsection{Picard Method} 
Let $(u^{n-1}_h,{\boldsymbol \sigma}^{n-1}_h)\in U_h\times {\boldsymbol \Sigma}_h$ be fixed. Given $u^{l-1}_h\in U_h$ (assuming $u^0_h=u^{n-1}_h$ at the first iteration step), find $(u^{l}_h,{\boldsymbol
\sigma}^{l}_h)\in U_h\times {\boldsymbol \Sigma}_h$ solving the linear coupled problem:
 \begin{equation}\label{linF1}
 \left\{
\begin{array}
[c]{lll}
\vspace{0.3 cm} \displaystyle\frac{1}{k}(u^l_h,\bar{u}_h)+ (\nabla
u^l_h, \nabla \bar{u}_h) + (u^{l-1}_h{\boldsymbol \sigma}^{l}_h,\nabla
\bar{u}_h)=\frac{1}{k}(u^{n-1}_h,\bar{u}_h), \ \forall \bar{u}_h\in U_h,\\
 \displaystyle\frac{1}{k}({\boldsymbol
\sigma}^l_h,\bar{\boldsymbol \sigma}_h)+ (B_h {\boldsymbol \sigma}^l_h,\bar{\boldsymbol \sigma}_h)- 2(u^{l-1}_h \nabla
u^l_h,\bar{\boldsymbol \sigma}_h)= \frac{1}{k}({\boldsymbol
\sigma}^{n-1}_h,\bar{\boldsymbol \sigma}_h), \ \forall \bar{\boldsymbol
\sigma}_h\in  {\boldsymbol \Sigma}_h,
\end{array}
\right.
\end{equation}
until that the following stopping criterion be satisfied:
\begin{equation}\label{stopp}
\max\left\{\displaystyle\frac{\Vert
u^l_h - u^{l-1}_h\Vert_{0}}{\Vert u^{l-1}_h\Vert_{0}},\displaystyle\frac{\Vert
{\boldsymbol \sigma}^l_h - {\boldsymbol \sigma}^{l-1}_h\Vert_{0}}{\Vert
{\boldsymbol \sigma}^{l-1}_h\Vert_{0}}\right\}\leq tol.
\end{equation}

\begin{tma} \label{piclinteo}{\bf(Unconditional  Solvability) }
There exists a unique $(u^l_h,{\boldsymbol\sigma}^l_h)$ solution of (\ref{linF1}).
\end{tma}
\begin{proof}
Since (\ref{linF1}) can be rewritten as a square linear algebraic system, it suffices to prove uniqueness. Let $(u^l_{h,1},{\boldsymbol
\sigma}^l_{h,1}),(u^l_{h,2},{\boldsymbol \sigma}^l_{h,2})\in U_h\times {\boldsymbol \Sigma}_h$ be two possible solutions of (\ref{linF1}).
Then defining $u^l_h=u^l_{h,1}-u^l_{h,2}$ and ${\boldsymbol
\sigma}^l_h={\boldsymbol \sigma}^l_{h,1}-{\boldsymbol \sigma}^l_{h,2}$,
 one has 
\begin{equation}\label{lp001}
\displaystyle\frac{1}{k}(u^l_h,\bar{u}_h)+ (\nabla u^{l}_h,
\nabla\bar{u}_h)+(u^{l-1}_h{\boldsymbol \sigma}^l_h,\nabla \bar{u}_h)=0, \ \forall
\bar{u}_h\in U_h,
\end{equation}
\begin{equation}\label{lp002}
\displaystyle\frac{1}{k}({\boldsymbol \sigma}^l_h,\bar{\boldsymbol
\sigma}_h)+ ( B_h {\boldsymbol \sigma}^l_h,\bar{\boldsymbol \sigma}_h) - 2(u^{l-1}_h\nabla u^l_h,\bar{\boldsymbol \sigma}_h)=0,\ \forall \bar{\boldsymbol
\sigma}_h\in {\boldsymbol \Sigma}_h.
\end{equation}
Taking $\bar{u}_h=u^l_h$ and $\bar{\boldsymbol \sigma}_h=\displaystyle
\frac{1}{2}{\boldsymbol \sigma}^l_h$ in  (\ref{lp001}) and (\ref{lp002}), and adding the resulting equations, the terms $(u^{l-1}_h\nabla
u^l_h,{\boldsymbol \sigma}^l_h)$ cancel, obtaining
\begin{equation*}\label{lp003}
\frac{1}{2k} \Vert (u^l_h,{\boldsymbol
\sigma}^l_h)\Vert_{0}^2+ \frac{1}{2}\Vert (\nabla
u^l_h,
{\boldsymbol \sigma}^l_h)\Vert_{L^2\times H^1}^2 \leq0,
\end{equation*}
hence $\Vert (u^l_h,{\boldsymbol \sigma}^l_h)\Vert_{1}=0$, which implies $u^l_{h,1}=u^l_{h,2}$ and
${\boldsymbol \sigma}^l_{h,1}={\boldsymbol \sigma}^l_{h,2}$.
\end{proof}

\begin{tma}\label{TCPM} {\bf (Local uniqueness of scheme \textbf{US} and Convergence of Picard's method)}
Given $(u^{n-1}_h,{\boldsymbol \sigma}^{n-1}_h)$, there exists $r>0$ (large enough) such that if 
\begin{equation}\label{restk01}
\displaystyle k \Vert (u^{n-1}_h,{\boldsymbol\sigma}^{n-1}_h)\Vert_{1}^4 \ \ \ \mbox{ and } \ \ \ k \,r^4 \quad \hbox{are small enough},
\end{equation}
then the scheme \textbf{US} has a unique solution $(u^n_h,{\boldsymbol\sigma}^n_h)$ in $\overline{B}_r((u_h^{n-1},{\boldsymbol \sigma}^{n-1}_h)):=\{(u,{\boldsymbol\sigma})\in U_h \times {\boldsymbol\Sigma}_h: \Vert (u - u^{n-1}_h, {\boldsymbol \sigma} - {\boldsymbol \sigma}^{n-1}_h)\Vert_{1}\leq r\}$. Moreover, the sequence of solutions $\{u^l_h,{\boldsymbol\sigma}^l_h\}_{l\geq0}$ of the iterative algorithm (\ref{linF1}) converges to $(u^n_h,{\boldsymbol\sigma}^n_h)$ strongly in $H^1(\Omega)$.
\end{tma}
\begin{proof}
Let the operator
$R:U_h\rightarrow U_h$ be given by
$R(\widetilde{u})=u$, where
$(u,{\boldsymbol\sigma})$ satisfies (\ref{linF1}) changing 
$u^{l-1}_h$ by $\widetilde{u}$ and
$(u^{l}_h,{\boldsymbol\sigma}^l_h)$ by $(u,{\boldsymbol\sigma})$, that is,
\begin{equation}\label{linFB1}
\displaystyle\frac{1}{k}(u,\bar{u}_h)+ (\nabla u, \nabla\bar{u}_h) +
(\widetilde{u}{\boldsymbol \sigma},\nabla
\bar{u}_h)=\frac{1}{k}(u^{n-1}_h,\bar{u}_h), \ \forall \bar{u}_h\in U_h,
\end{equation}
\vspace{-.5 cm}
\begin{equation}\label{linFB2}
\displaystyle\frac{1}{k}({\boldsymbol \sigma},\bar{\boldsymbol
\sigma}_h)+ ( B_h {\boldsymbol
\sigma},\bar{\boldsymbol \sigma}_h) - 2(\widetilde{u} \nabla
u,\bar{\boldsymbol \sigma}_h)= \frac{1}{k}({\boldsymbol
\sigma}^{n-1}_h,\bar{\boldsymbol \sigma}_h), \ \forall \bar{\boldsymbol
\sigma}_h\in {\boldsymbol \Sigma}_h.
\end{equation}
From Theorem \ref{piclinteo},  for any
$\widetilde{u}\in U_h$ there exists a unique $(u,{\boldsymbol \sigma})\in
U_h \times {\boldsymbol \Sigma}_h$ solution of
(\ref{linFB1})-(\ref{linFB2}). Thus, $R$ is well defined. 
Now, before proving that $R$ is contractive, we will construct a ball $\overline{B}_r(u^{n-1}_h)=\{u\in U_h : \Vert u - u^{n-1}_h\Vert_{1}\leq r\}\subset U_h$ such that $R(\overline{B}_r(u^{n-1}_h)) \subseteq \overline{B}_r(u^{n-1}_h)$. 
In order to define $r$,  one considers $w=u-u^{n-1}_h$ and ${\boldsymbol\tau}={\boldsymbol \sigma} - {\boldsymbol \sigma}^{n-1}_h$. 
Then, from (\ref{linFB1})-(\ref{linFB2}) one has
\begin{equation}\label{inv01}
\displaystyle\frac{1}{k}(w,\bar{u}_h)+ (\nabla  w, \nabla \bar{u}_h)=
-(\widetilde{u}{\boldsymbol \tau},\nabla
\bar{u}_h)-(\nabla u^{n-1}_h, \nabla\bar{u}_h)-(\widetilde{u}{\boldsymbol \sigma}^{n-1}_h,\nabla
\bar{u}_h), \ \forall \bar{u}_h\in U_h,
\end{equation}
\vspace{-.5 cm}
\begin{equation}\label{inv02}
\displaystyle\frac{1}{k}({\boldsymbol \tau},\bar{\boldsymbol
\sigma}_h)+ ( B_h {\boldsymbol
\tau},\bar{\boldsymbol \sigma}_h)=  2(\widetilde{u} \nabla w,\bar{\boldsymbol \sigma}_h)- (B_h{\boldsymbol
\sigma}^{n-1}_h,\bar{\boldsymbol \sigma}_h)+ 2(\widetilde{u} \nabla
u^{n-1}_h,\bar{\boldsymbol \sigma}_h) , \ \forall \bar{\boldsymbol
\sigma}_h\in {\boldsymbol \Sigma}_h.
\end{equation}
Taking $\bar{u}_h=w$ and $\displaystyle\bar{\boldsymbol \sigma}_h=\frac{1}{2}{\boldsymbol \tau}$ in (\ref{inv01})-(\ref{inv02}) and adding, the terms $(\widetilde{u} \nabla
w,{\boldsymbol \tau})$ cancel, and using the fact that $\displaystyle \int_\Omega w=0$ as well as the 3D interpolation inequality (\ref{in3D}),  it holds
\begin{eqnarray*}\label{inv0001}
&\displaystyle\frac{1}{2k} \Vert (w,{\boldsymbol
\tau})\Vert_{0}^2&\!\!\!\!+ \frac{1}{2}\Vert 
(w,{\boldsymbol
\tau})\Vert_{1}^2 \leq \frac{1}{8} \Vert 
(w,{\boldsymbol \tau})\Vert_{1}^2+ C \Vert 
(u^{n-1}_h, {\boldsymbol \sigma}^{n-1}_h)\Vert_{1}^2 \nonumber\\
&&\!\!\!\!\!\!\!\!\!\!\!\! +\frac{1}{8}\Vert \widetilde{u}-u^{n-1}_h\Vert_{1}^2+\frac{1}{8}\Vert u^{n-1}_h\Vert_{1}^2+ \frac{1}{8}\Vert
(w,{\boldsymbol\tau}) \Vert_{1}^2 +C \Vert (u_h^{n-1},{\boldsymbol \sigma}^{n-1}_h)\Vert_{1}^4 \Vert
(w,{\boldsymbol\tau}) \Vert_{0}^2.
\end{eqnarray*}
Therefore, 
\begin{eqnarray}\label{inv0002}
\left[\displaystyle\frac{1}{2k} - C \Vert (u_h^{n-1},{\boldsymbol \sigma}^{n-1}_h)\Vert_{1}^4\right] \Vert (w, {\boldsymbol
\tau})\Vert_{0}^2+ \frac{1}{4}\Vert 
(w,{\boldsymbol \tau})\Vert_{1}^2  \leq   C \Vert 
(u^{n-1}_h,
{\boldsymbol \sigma}^{n-1}_h)\Vert_{1}^2+\frac{1}{8}\Vert \widetilde{u}-u^{n-1}_h\Vert_{1}^2.
\end{eqnarray}
Thus, if $\displaystyle k<\frac{1}{2C \Vert (u_h^{n-1},{\boldsymbol \sigma}^{n-1}_h)\Vert_{1}^4}$, from (\ref{inv0002}), one concludes
\begin{equation}\label{inv0003}
\Vert 
(w,{\boldsymbol \tau})\Vert_{1}^2 \leq   C\Vert 
(u^{n-1}_h,
{\boldsymbol \sigma}^{n-1}_h)\Vert_{1}^2+\frac{1}{2}\Vert \widetilde{u}-u^{n-1}_h\Vert_{1}^2.
\end{equation}
Then, choosing $r>0$ large enough such that 
\begin{equation}\label{Pmd}
C\Vert 
(u^{n-1}_h,
{\boldsymbol \sigma}^{n-1}_h)\Vert_{1}^2\leq \displaystyle\frac{1}{2}r^2,
\end{equation}
from (\ref{inv0003}) one deduces that $R(\overline{B}_r(u^{n-1}_h)) \subseteq \overline{B}_r(u^{n-1}_h)$. 
Then,  the restriction of $R$ to $\overline{B}_r(u^{n-1}_h)$ is taken, that is, $R_r:\overline{B}_r(u^{n-1}_h)\rightarrow \overline{B}_r(u^{n-1}_h)$. Let us prove that $R_r$ is contractive. Let
$\widetilde{u}_1,\widetilde{u}_2\in \overline{B}_r(u^{n-1}_h)$,
 and $(u_1,{\boldsymbol \sigma}_1)$ and
$(u_2,{\boldsymbol \sigma}_2)$ solutions of (\ref{linFB1})-(\ref{linFB2}) related to $\widetilde{u}_1$ and $\widetilde{u}_2$ respectively (i.e., $R_r(\widetilde{u}_1)=u_1$ and $R_r(\widetilde{u}_2)=u_2$). 
Then, from
(\ref{linFB1})-(\ref{linFB2}) one has that $(u_1
-u_2,{\boldsymbol\sigma}_1 - {\boldsymbol\sigma}_2)\in
U_h \times {\boldsymbol \Sigma}_h$ satisfies
\begin{equation*}\label{linFBC1}
\displaystyle\frac{1}{k}(u_1 -u_2,\bar{u}_h)+ (\nabla (u_1
-u_2), \nabla \bar{u}_h) + (\widetilde{u}_1({\boldsymbol \sigma}_1 -
{\boldsymbol\sigma}_2),\nabla
\bar{u}_h)+((\widetilde{u}_1-\widetilde{u}_2)
{\boldsymbol\sigma}_2,\nabla \bar{u}_h)=0, \ \forall \bar{u}_h\in U_h,
\end{equation*}
\vspace{-.5 cm}
\begin{equation*}\label{linFBC2}
\displaystyle\frac{1}{k}({\boldsymbol\sigma}_1 -
{\boldsymbol\sigma}_2,\bar{\boldsymbol \sigma}_h)+
( B_h ({\boldsymbol\sigma}_1 -
{\boldsymbol\sigma}_2),\bar{\boldsymbol \sigma}_h) - 2(\widetilde{u}_1 \nabla (u_1 - u_2),\bar{\boldsymbol
\sigma}_h)- 2((\widetilde{u}_1-\widetilde{u}_2) \nabla
u_2,\bar{\boldsymbol \sigma}_h)= 0, \ \forall \bar{\boldsymbol
\sigma}_h\in {\boldsymbol \Sigma}_h.
\end{equation*}
Taking  $\bar{u}_h=u_1-u_2$,
$\displaystyle\bar{\boldsymbol\sigma}_h=\frac{1}{2}({\boldsymbol\sigma}_1
- {\boldsymbol\sigma}_2)$ and adding, the terms $(\widetilde{u}_1({\boldsymbol \sigma}_1 -
{\boldsymbol\sigma}_2),\nabla(u_1 - u_2))$ cancel, and using the H\"older and Young inequalities, the 3D interpolation inequality (\ref{in3D}) and taking into account that $\int_\Omega u_1 - u_2 =0$,  one obtains
\begin{eqnarray*}
&\displaystyle\frac{1}{2k}&\!\!\! \Vert (u_1 -
u_2, {\boldsymbol\sigma}_1 - {\boldsymbol\sigma}_2) \Vert_{0}^2+ \Vert  u_1 -
u_2\Vert_{1}^2+
\frac{1}{2}\Vert {\boldsymbol\sigma}_1 -
{\boldsymbol\sigma}_2\Vert_{1}^2 \nonumber\\
&&\!\!\!\leq C\Vert
\widetilde{u}_1-\widetilde{u}_2\Vert_{1}(
\Vert{\boldsymbol\sigma}_2\Vert_{1}\Vert u_1 -
u_2\Vert_{L^3}+
\Vert 
u_2\Vert_{1}\Vert{\boldsymbol\sigma}_1 -
{\boldsymbol\sigma}_2\Vert_{L^3})\nonumber\\
&&\!\!\!\leq
\displaystyle\frac{1}{4}\Vert \widetilde{u}_1-\widetilde{u}_2\Vert_{1}^2+\frac{1}{2}\Vert u_1 -
u_2\Vert_{1}^2 +\displaystyle\frac{1}{4}\Vert{\boldsymbol\sigma}_1 -
{\boldsymbol\sigma}_2\Vert_{1}^2 +C\Vert (u_1 -
u_2, {\boldsymbol\sigma}_1 - {\boldsymbol\sigma}_2) \Vert_{0}^2\Vert(u_2,{\boldsymbol\sigma}_2)\Vert_{1}^4.
\end{eqnarray*}
Therefore, 
\begin{eqnarray}\label{linFBC10}
&\displaystyle\frac{1}{k} &\!\!\!\!\!\Vert (u_1 -
u_2, {\boldsymbol\sigma}_1 - {\boldsymbol\sigma}_2) \Vert_{0}^2+ \Vert u_1 -
u_2\Vert_{1}^2+\frac{1}{2}\Vert
{\boldsymbol\sigma}_1 - {\boldsymbol\sigma}_2\Vert_{1}^2\nonumber\\
&&\!\!\!\!\!\leq \displaystyle\frac{1}{2}\Vert \widetilde{u}_1-\widetilde{u}_2\Vert_{1}^2+C\Vert (u_1 -
u_2, {\boldsymbol\sigma}_1 - {\boldsymbol\sigma}_2) \Vert_{0}^2\Vert(u_2,{\boldsymbol\sigma}_2)\Vert_{1}^4 .
\end{eqnarray}
Since  (\ref{inv0003}) and (\ref{Pmd}) imply $\Vert (u_2,{\boldsymbol\sigma}_2)\Vert^4_{1}\leq C(r^4+\Vert (u^{n-1}_h,{\boldsymbol\sigma}^{n-1}_h)\Vert^4_{1})$, then if  $\displaystyle \frac{1}{2k}>C r^4$ and $\displaystyle\frac{1}{2k}>C\Vert (u^{n-1}_h,{\boldsymbol\sigma}^{n-1}_h)\Vert_{1}^4$,   one has from (\ref{linFBC10}):
\begin{equation*}
\Vert R_r(\widetilde{u}_1) -
R_r(\widetilde{u}_2)\Vert_{1}^2\leq \frac{1}{2}
\Vert \widetilde{u}_1-\widetilde{u}_2\Vert_{1}^2,
\end{equation*}
i.e.~$R_r$ is contractive. Then, the Banach fixed point theorem implies the existence of a unique fixed point of $R_r$, $R_r(u)=u$. 
Thus, $(u,{\boldsymbol \sigma})$ is the unique solution of the scheme \textbf{US}
with $u\in\overline{B}_r(u^{n-1}_h)$. 
 Additionally,  the sequence $\{u^l_h,{\boldsymbol\sigma}^l_h\}_{l\geq0}$ of the iterative algorithm (\ref{linF1}) converges to the solution $(u^n_h,{\boldsymbol\sigma}^n_h)$.
\end{proof}

\begin{obs}
In the case of 2D domains, since estimate (\ref{strong01}) holds, then the restriction (\ref{restk01})$_1$ can be relaxed to $k\le K_0$, where $K_0$ is a constant depending on data $(\Omega,u_0,{\boldsymbol\sigma}_0)$, but independent of $(k,h)$ and $n$.
\end{obs}
\begin{obs} 
 Notice that the restriction (\ref{restk01})$_1$  is equivalent to (\ref{uniq01}). Therefore, under  the hypotheses of Theorem \ref{erteo} and  arguing as in Remark~\ref{remH2},  the conclusion of Theorem \ref{TCPM} remains true only assuming $k$ small enough.
\end{obs}

\subsection{Newton's Method}
Let $(u^{n-1}_h,{\boldsymbol \sigma}^{n-1}_h)\in U_h\times {\boldsymbol \Sigma}_h$ be fixed. Given $(u^{l-1}_h,{\boldsymbol \sigma}^{l-1}_h)\in U_h\times {\boldsymbol \Sigma}_h$,  
find $(u^{l}_h,{\boldsymbol \sigma}^{l}_h)\in U_h\times {\boldsymbol \Sigma}_h$ solving the linear coupled problem:
 \begin{equation}\label{NM01}
 \left\{
\begin{array}
[c]{lll}
\vspace{0.3 cm}\displaystyle\frac{1}{k}(u^l_h,\bar{u}_h)+ (\nabla 
u^l_h, \nabla\bar{u}_h) + (u^{l-1}_h{\boldsymbol \sigma}^{l}_h,\nabla
\bar{u}_h)+(u^{l}_h{\boldsymbol \sigma}^{l-1}_h,\nabla
\bar{u}_h)=\frac{1}{k}(u^{n-1}_h,\bar{u}_h)+(u^{l-1}_h{\boldsymbol \sigma}^{l-1}_h,\nabla
\bar{u}_h),\\
\displaystyle\frac{1}{k}({\boldsymbol
\sigma}^l_h,\bar{\boldsymbol \sigma}_h)+ (B_h {\boldsymbol
\sigma}^l_h,\bar{\boldsymbol \sigma}_h)- 2(u^{l-1}_h \nabla
u^l_h,\bar{\boldsymbol \sigma}_h) \\ 
\hspace{3 cm}- 2(u^{l}_h \nabla
u^{l-1}_h,\bar{\boldsymbol \sigma}_h)= \displaystyle\frac{1}{k}({\boldsymbol
\sigma}^{n-1}_h,\bar{\boldsymbol \sigma}_h)- 2(u^{l-1}_h \nabla
u^{l-1}_h,\bar{\boldsymbol \sigma}_h),
\end{array}
\right.
\end{equation} 
 for all $(\bar{u}_h,\bar{\boldsymbol\sigma}_h)\in  U_h\times{\boldsymbol \Sigma}_h$. 
 Iterations will repeat until the stopping criterion  \eqref{stopp} be satisfied.\\

The following result will be applied to obtain the convergence of Newton's method \eqref{NM01}.
\begin{lem}\label{lemconh1}
Let $X$ be a Banach space and consider a sequence $\{e_l\}_{l\geq0}\subseteq X$, such that
\begin{equation*}
\Vert e_l\Vert_{X}^2\leq C\left(\Vert e_{l-1}\Vert_{X}^2\right)^2, \ \ \forall l\geq 1\ \ \mbox{ and } \ \ \Vert e_0\Vert_{X}^2 \ \mbox{ is small enough.}
\end{equation*}
Then, $e_l$ converges to $0$ as $l\rightarrow +\infty$ in the $X$-norm.
\end{lem}
\begin{tma}\label{CNMtma}
{\bf (Conditional convergence of Newton's method)} Let $(u^n_h,{\boldsymbol\sigma}^n_h)$ be a fixed solution of the scheme \textbf{US} and let $(u^l_h,{\boldsymbol\sigma}^l_h)$ be any solution of (\ref{NM01}). There exists $\delta_0>0$ small enough such that if 
\begin{equation}\label{cmn0a}
 \Vert (e_u^{0},e_{\boldsymbol \sigma}^{0})\Vert_{1}^2\leq \delta_0, \ \  \ \ \displaystyle k\Vert (u^{n}_h,{\boldsymbol \sigma}^{n}_h)\Vert_{1}^4 \ \ \mbox{ and }  \ \ \displaystyle k (\delta_0)^2 \quad \hbox{are small enough,} 
\end{equation}
 then $\{u^l_h,{\boldsymbol\sigma}^l_h\}_{l\geq0}$ converges to $(u^n_h,{\boldsymbol\sigma}^n_h)$ in the $H^1(\Omega)$-norm as $l\rightarrow +\infty$.
\end{tma}
\begin{proof}
We can rewrite problem (\ref{modelf02}) in a vectorial way,
\begin{equation}\label{CMN01}
(0,0)=\langle \mathbf{F}(u^n_h,{\boldsymbol\sigma}^n_h), (\bar{u}_h,\bar{\boldsymbol\sigma}_h)\rangle=\left( \langle F_1(u^n_h,{\boldsymbol\sigma}^n_h),\bar{u}_h\rangle, \langle F_2(u^n_h,{\boldsymbol\sigma}^n_h), \bar{\boldsymbol\sigma}_h\rangle\right),
\end{equation}
where each $F_i(u^n_h,{\boldsymbol\sigma}^n_h)$ corresponds with the equation (\ref{modelf02})$_i$ ($i=1,2$). Therefore, Newton's method (\ref{NM01}) reads
\begin{equation*}
\langle \mathbf{F}'(u^{l-1}_h,{\boldsymbol\sigma}^{l-1}_h)(u^{l}_h-u^{l-1}_h,{\boldsymbol\sigma}^l_h - {\boldsymbol\sigma}^{l-1}_h),(\bar{u}_h,\bar{\boldsymbol\sigma}_h)\rangle =- \langle \mathbf{F}(u^{l-1}_h,{\boldsymbol\sigma}^{l-1}_h), (\bar{u}_h,\bar{\boldsymbol\sigma}_h)\rangle,
\end{equation*}
which can be rewritten as 
\begin{eqnarray}\label{Nal}
&&\!\!\!\!\!\!\!\!(0,0)=( \langle F_1(u^{l-1}_h,{\boldsymbol\sigma}^{l-1}_h),\bar{u}_h\rangle, \langle F_2(u^{l-1}_h,{\boldsymbol\sigma}^{l-1}_h), \bar{\boldsymbol\sigma}_h\rangle) \nonumber\\
&& \!\!\!\! + ( \langle F'_1(u^{l-1}_h,{\boldsymbol\sigma}^{l-1}_h)(u^{l}_h-u^{l-1}_h,{\boldsymbol\sigma}^l_h - {\boldsymbol\sigma}^{l-1}_h),\bar{u}_h\rangle, \langle F'_2(u^{l-1}_h,{\boldsymbol\sigma}^{l-1}_h)(u^{l}_h-u^{l-1}_h,{\boldsymbol\sigma}^l_h - {\boldsymbol\sigma}^{l-1}_h),\bar{\boldsymbol\sigma}_h\rangle).\ \ \ \ \ \ \ 
\end{eqnarray}
Moreover, from a vectorial Taylor's formula of $\mathbf{F}(u^n_h,{\boldsymbol\sigma}^n_h)$ with center at $(u^{l-1}_h,{\boldsymbol\sigma}^{l-1}_h)$, and using (\ref{CMN01}),  one has that
\begin{eqnarray}\label{CMN03}
&(0,0)&\!\!\!=\displaystyle\left( \langle F_1(u^n_h,{\boldsymbol\sigma}^n_h),\bar{u}_h\rangle, \langle F_2(u^n_h,{\boldsymbol\sigma}^n_h),\bar{\boldsymbol\sigma}_h\rangle \right)\nonumber\\
&&\!\!\!\!\!\! \!\! = \left( \langle F_1(u^{l-1}_h,{\boldsymbol\sigma}^{l-1}_h),\bar{u}_h\rangle, \langle F_2(u^{l-1}_h,{\boldsymbol\sigma}^{l-1}_h),\bar{\boldsymbol\sigma}_h\rangle \right)\nonumber\\
&&\!\!\!\!\!\! \!\! +  \left( \langle F'_1(u^{l-1}_h,{\boldsymbol\sigma}^{l-1}_h)(u^{n}_h-u^{l-1}_h,{\boldsymbol\sigma}^n_h - {\boldsymbol\sigma}^{l-1}_h),\bar{u}_h\rangle, \langle F'_2(u^{l-1}_h,{\boldsymbol\sigma}^{l-1}_h)(u^{n}_h-u^{l-1}_h,{\boldsymbol\sigma}^n_h - {\boldsymbol\sigma}^{l-1}_h),\bar{\boldsymbol\sigma}_h\rangle \right)\nonumber\\
&&  \!\!\!\!\!\! \!\!+ \displaystyle\frac{1}{2} \Big( \langle(u^{n}_h-u^{l-1}_h,{\boldsymbol\sigma}^n_h - {\boldsymbol\sigma}^{l-1}_h)^t F''_1(u^{n+\varepsilon},{\boldsymbol\sigma}^{n+\varepsilon})(u^{n}_h-u^{l-1}_h,{\boldsymbol\sigma}^n_h - {\boldsymbol\sigma}^{l-1}_h),\bar{u}_h\rangle, \nonumber\\
&&\!\!\!\!\!  \mbox{ }  \mbox{ }   \mbox{ }  \langle (u^{n}_h-u^{l-1}_h,{\boldsymbol\sigma}^n_h - {\boldsymbol\sigma}^{l-1}_h)^t F''_2(u^{n+\varepsilon},{\boldsymbol\sigma}^{n+\varepsilon})(u^{n}_h-u^{l-1}_h,{\boldsymbol\sigma}^n_h - {\boldsymbol\sigma}^{l-1}_h),\bar{\boldsymbol\sigma}_h\rangle \Big),
\end{eqnarray}
where $u^{n+\varepsilon}=\varepsilon u^n_h + (1-\varepsilon)u^{l-1}_h$, ${\boldsymbol\sigma}^{n+\varepsilon}=\varepsilon {\boldsymbol\sigma}^n_h + (1-\varepsilon) {\boldsymbol\sigma}^{l-1}_h$, and $F'_i$ and $F''_i$ denote the Jacobian and the Hessian of $F_i$ ($i=1,2$), respectively. 
Therefore, denoting by $e_u^l=u^n_h - u^l_h$ and $e_{\boldsymbol\sigma}^l={\boldsymbol\sigma}^n_h-{\boldsymbol\sigma}^l_h$, from (\ref{Nal})-(\ref{CMN03}), we deduce 
\begin{eqnarray}\label{CMN05}
&&\displaystyle \left\langle \frac{\partial F_1}{\partial u} (u^{l-1}_h,{\boldsymbol\sigma}^{l-1}_h) (e_u^l) +\frac{\partial F_1}{\partial {\boldsymbol\sigma}} (u^{l-1}_h,{\boldsymbol\sigma}^{l-1}_h) (e_{\boldsymbol\sigma}^l), \bar{u}_h\right\rangle\nonumber\\
&& \hspace{3,5 cm} =-\displaystyle\frac{1}{2} \langle(e_u^{l-1},e_{\boldsymbol\sigma}^{l-1})^t F''_1(u^{n+\varepsilon},{\boldsymbol\sigma}^{n+\varepsilon})(e_u^{l-1},e_{\boldsymbol\sigma}^{l-1}),\bar{u}_h\rangle,
\end{eqnarray}
\begin{eqnarray}\label{CMN06}
&&\displaystyle \left\langle \frac{\partial F_2}{\partial u} (u^{l-1}_h,{\boldsymbol\sigma}^{l-1}_h) (e_u^l) +\frac{\partial F_2}{\partial {\boldsymbol\sigma}} (u^{l-1}_h,{\boldsymbol\sigma}^{l-1}_h) (e_{\boldsymbol\sigma}^l), \bar{\boldsymbol\sigma}_h\right\rangle\nonumber\\
&& \hspace{3,5 cm}=-\displaystyle\frac{1}{2} \langle(e_u^{l-1},e_{\boldsymbol\sigma}^{l-1})^t F''_2(u^{n+\varepsilon},{\boldsymbol\sigma}^{n+\varepsilon})(e_u^{l-1},e_{\boldsymbol\sigma}^{l-1}),\bar{\boldsymbol\sigma}_h\rangle.
\end{eqnarray}
Thus, from (\ref{CMN05})-(\ref{CMN06}) and taking into account that $F''_i$ are constant matrices, we arrive at
\begin{equation}\label{CMN07}
\displaystyle\frac{1}{k}(e_u^l,\bar{u}_h)+ (\nabla
e_u^l, \nabla\bar{u}_h)+ (e_u^{l}{\boldsymbol \sigma}^{l-1}_h,\nabla
\bar{u}_h)+(u^{l-1}_he_{\boldsymbol \sigma}^{l},\nabla
\bar{u}_h)=-(e_u^{l-1}e_{\boldsymbol \sigma}^{l-1},\nabla
\bar{u}_h), \ \ \forall \bar{u}_h\in U_h,
\end{equation}
\begin{equation}\label{CMN08}
\displaystyle\frac{1}{k}(e_{\boldsymbol
\sigma}^l,\bar{\boldsymbol \sigma}_h) + (B_h e_{\boldsymbol
\sigma}^l,\bar{\boldsymbol \sigma}_h)
+2(u^{l-1}_h e_u^l,\nabla\cdot
\bar{\boldsymbol \sigma}_h)=-(\vert e_u^{l-1}\vert^2,\nabla\cdot
\bar{\boldsymbol \sigma}_h), \ \ \forall \bar{\boldsymbol\sigma}_h\in {\boldsymbol \Sigma}_h.
\end{equation}
Taking  $\bar{u}_h=e_u^l$ and $\bar{\boldsymbol \sigma}_h=e_{\boldsymbol \sigma}^l$ in (\ref{CMN07}) and (\ref{CMN08}) respectively, taking into account that $\displaystyle\int_\Omega e_u^l=0$ and using the H\"older and Young inequalities as well as the 3D interpolation inequality (\ref{in3D}), 
\begin{equation}\label{CMN09}
\displaystyle\frac{1}{k}\Vert
(e_u^l,e_{\boldsymbol
\sigma}^l)\Vert_{0}^2+\Vert  (e_u^l ,e_{\boldsymbol
\sigma}^l)\Vert_{1}^2 \leq \frac{1}{2}\Vert  (e_u^l ,e_{\boldsymbol
\sigma}^l)\Vert_{1}^2 + C\Vert (e_u^l ,e_{\boldsymbol
\sigma}^l)\Vert_{0}^{2}\Vert( u^{l-1}_h,{\boldsymbol \sigma}^{l-1}_h)\Vert_{1}^4 +C \Vert (e_u^{l-1},e_{\boldsymbol \sigma}^{l-1})\Vert_{1}^4. 
\end{equation}
In order to use an inductive strategy, the following hypothesis will be assumed
\begin{equation*}
\Vert (e_u^{l-1},e_{\boldsymbol \sigma}^{l-1})\Vert_{1}^2\leq \delta_0,
\end{equation*}
which implies that
\begin{equation}\label{CMN10}
\Vert (u^{l-1}_h,{\boldsymbol \sigma}^{l-1}_h)\Vert_{1}\leq \Vert (u^{n}_h,{\boldsymbol \sigma}^{n}_h)\Vert_{1}+\sqrt{\delta_0},
\end{equation}
where $\delta_0>0$ is a small enough constant. 
Therefore, from (\ref{CMN09})-(\ref{CMN10}), one has
\begin{equation}\label{CMN11}
\left(\frac{1}{k}-C(\Vert (u^{n}_h,{\boldsymbol \sigma}^{n}_h)\Vert_{1}^4+(\delta_0)^2)\right)\Vert
(e_u^l,e_{\boldsymbol
\sigma}^l)\Vert_{0}^2+\frac{1}{2}\Vert (e_u^l, e_{\boldsymbol
\sigma}^l)\Vert_{1}^2\leq C\left(\Vert (e_u^{l-1}, e_{\boldsymbol \sigma}^{l-1})\Vert_{1}^2\right)^2.
\end{equation}
Thus, if $\displaystyle\frac{1}{2k}>C\Vert (u^{n}_h,{\boldsymbol \sigma}^{n}_h)\Vert_{1}^4$ and $\displaystyle\frac{1}{2k}>C(\delta_0)^2$ (which is possible owing to (\ref{cmn0a})$_2$ and (\ref{cmn0a})$_3$), one has from (\ref{CMN11}) 
\begin{equation}\label{CMN12}
\Vert (e_u^l, e_{\boldsymbol
\sigma}^l)\Vert_{1}^2\leq C\left(\Vert (e_u^{l-1}, e_{\boldsymbol \sigma}^{l-1})\Vert_{1}^2\right)^2.
\end{equation}
Therefore, choosing $\delta_0$ small enough such that $\delta_0C\leq 1$, the inequality $\Vert (e_u^{l},e_{\boldsymbol \sigma}^{l})\Vert_{1}^2\leq \delta_0$ holds. 
Indeed,  assuming $\Vert (e_u^{0},e_{\boldsymbol \sigma}^{0})\Vert_{1}^2\leq \delta_0$, 
 the following recurrence expression is obtained
\begin{equation}\label{rec01}
\Vert (e_u^{l},e_{\boldsymbol \sigma}^{l})\Vert_{1}^2\leq \Vert (e_u^{l-1},e_{\boldsymbol \sigma}^{l-1})\Vert_{1}^2\leq \cdot\cdot\cdot\leq\Vert (e_u^{0},e_{\boldsymbol \sigma}^{0})\Vert_{1}^2\leq \delta_0. 
\end{equation}
Hence, from (\ref{CMN12})  the  hypotheses of Lemma \ref{lemconh1} are satisfied, and we conclude the convergence of $(u^l_h,{\boldsymbol\sigma}^l_h)$ to $(u^n_h,{\boldsymbol\sigma}^n_h)$ in the $H^1(\Omega)$-norm.
\end{proof}
\begin{obs}\label{NMb}
If (\ref{strong01}) is satisfied (recall that this estimate holds, at least, in 2D domains), we can determine $\delta_0$ in terms of $k$. Indeed, from (\ref{stdta}), we have that 
\begin{equation*}
\Vert (e_u^{0},e_{\boldsymbol \sigma}^{0})\Vert_{1}^2=\Vert (u^n_h - u^{n-1}_h,{\boldsymbol \sigma}^n_h-{\boldsymbol \sigma}^{n-1}_h)\Vert_{1}^2\leq {\color{blue} K_2 k,}
\end{equation*}
where $K_2$ is the constant appearing in \eqref{stdta}. Therefore, we can consider {\color{blue}$\delta_0:=K_2 k$}. Then, the hypotheses (\ref{cmn0a}) in Theorem \ref{CNMtma} are only imposed on $k$, and (\ref{cmn0a})$_2$ is reduced to $k\leq K_0$, where $K_0$ is a constant depending on data $(\Omega,u_0,{\boldsymbol\sigma}_0)$, but independent of $(k,h)$ and $n$.
\end{obs}

\begin{obs} Since restriction (\ref{cmn0a})$_2$ is equivalent to  (\ref{uniq01}), 
analogously as in Remark \ref{inddom01}, under  the hypotheses of Theorem \ref{erteo}, the conclusion of Theorem \ref{CNMtma} remains true  assuming $k$ small enough, (\ref{cmn0a})$_1$ and (\ref{cmn0a})$_3$.
\end{obs}

Now, observe that from (\ref{rec01}), the following estimate for $(u^l_h,{\boldsymbol\sigma}^l_h)$ solution of (\ref{NM01}) is obtained:
\begin{equation}\label{rec02}
\Vert (u^{l}_h,{\boldsymbol \sigma}^{l}_h)\Vert_{1}\leq \Vert (u^{n}_h,{\boldsymbol \sigma}^{n}_h)\Vert_{1}+\sqrt{\delta_0}, \ \ \forall l\geq 0.
\end{equation}

Then, using the above estimate,  the conditional unique solvability of (\ref{NM01}) will be proved.
\begin{tma} \label{piclinteoNewt}{\bf (Conditional unique solvability) }
Assume (\ref{cmn0a}). Then there exists a unique $(u^l_h,{\boldsymbol\sigma}^l_h)$ solution of (\ref{NM01}).
\end{tma}
\begin{proof}
By linearity, it suffices to prove uniqueness of solution of (\ref{NM01}). Let $(u^l_{h,1},{\boldsymbol
\sigma}^l_{h,1}),(u^l_{h,2},{\boldsymbol \sigma}^l_{h,2})\in U_h\times  {\boldsymbol \Sigma}_h$ be two solutions of (\ref{NM01}). Then, denoting $u^l_h=u^l_{h,1}-u^l_{h,2}$ and ${\boldsymbol
\sigma}^l_h={\boldsymbol \sigma}^l_{h,1}-{\boldsymbol \sigma}^l_{h,2}$,
\begin{equation}\label{NM03}
\frac{1}{k}(u^l_h,\bar{u}_h)+ (\nabla u^{l}_h, \nabla\bar{u}_h)+ (u^{l-1}_h{\boldsymbol \sigma}^{l}_h,\nabla
\bar{u}_h)+(u^{l}_h{\boldsymbol \sigma}^{l-1}_h,\nabla
\bar{u}_h)=0, \ \forall
\bar{u}_h\in U_h,
\end{equation}
\vspace{-.5 cm}
\begin{equation}\label{NM04}
\frac{1}{k}({\boldsymbol \sigma}^l_h,\bar{\boldsymbol
\sigma}_h)+ (B_h {\boldsymbol
\sigma}^l_h,\bar{\boldsymbol \sigma}_h)- 2(u^{l-1}_h \nabla
u^l_h,\bar{\boldsymbol \sigma}_h) - 2(u^{l}_h \nabla
u^{l-1}_h,\bar{\boldsymbol \sigma}_h)=0, \ \forall
\bar{\boldsymbol \sigma}_h\in {\boldsymbol \Sigma}_h.
\end{equation}
Taking $\bar{u}_h=u^l_h$ and $\displaystyle\bar{\boldsymbol \sigma}_h=\frac{1}{2}{\boldsymbol
\sigma}^l_h$ in (\ref{NM03})-(\ref{NM04}), taking into account that $\displaystyle\int_\Omega u^l_h=0$ and using the H\"older and Young inequalities 
and (\ref{in3D}), 
 one obtains 
\begin{equation*}
\displaystyle\frac{1}{2k}\Vert
(u^l_h,{\boldsymbol
\sigma}^l_h)\Vert_{0}^2+\frac{1}{2}\Vert (u^l_h, {\boldsymbol
\sigma}^l_h)\Vert_{1}^2 \leq \displaystyle\frac{1}{4}\Vert (u^l_h, {\boldsymbol
\sigma}^l_h)\Vert_{1}^2 +  C\Vert (u^{l-1}_h,{\boldsymbol \sigma}^{l-1}_h)\Vert_{1}^4 \Vert
(u^l_h,{\boldsymbol
\sigma}^l_h)\Vert_{0}^2,
\end{equation*}
which, using (\ref{rec02}) (recall that (\ref{rec02}) holds assuming (\ref{cmn0a})), implies that
\begin{equation}\label{NM06}
\left[\displaystyle\frac{1}{k}- C\Big(\Vert (u^{n}_h,{\boldsymbol \sigma}^{n}_h )\Vert_{1}^4+(\delta_0)^2 \Big)\right] \Vert
(u^l_h,{\boldsymbol
\sigma}^l_h)\Vert_{0}^2+\frac{1}{2} \Vert (u^l_h, {\boldsymbol
\sigma}^l_h)\Vert_{1}^2\leq 0.
\end{equation}
Therefore, assuming (\ref{cmn0a})$_{2-3}$, from (\ref{NM06}) we conclude that $\Vert (u^l_h,{\boldsymbol\sigma}^l_h)\Vert_{1}=0$, and therefore, $u^l_{h,1}=u^l_{h,2}$ and ${\boldsymbol\sigma}^l_{h,1}={\boldsymbol\sigma}^l_{h,2}$. Thus, there exists a unique $(u^l_h,{\boldsymbol\sigma}^l_h)$ solution of (\ref{NM01}).
\end{proof}

\section{Numerical results}
In this section, we consider the nonlinear scheme $\textbf{US}$ approximating \eqref{modelf01}-\eqref{modelf01eqv} with adequate right hand sides 
 corresponding to the exact solution
 $$u=e^{-t}(\cos(2\pi x) \cos(2\pi y) +2), \quad v=(1+\sin(t))(\cos(2\pi x) \cos(2\pi y) +2),
 $$
 $${\boldsymbol\sigma}=\nabla v=(1+sin(t)) (-2\pi sin(2\pi x)cos(2\pi y),-2\pi \sin(2\pi y)\cos(2\pi x)).
 $$
  In our computations, we take $\Omega=(0,1)^2$, and we use a uniform partition with $m+1$ nodes in each direction. We choose the spaces for $u$, ${\boldsymbol\sigma}$ and $v$, generated by $\mathbb{P}_1,\mathbb{P}_1,\mathbb{P}_{2}$-continuous FE, respectively. The linear iterative method used 
  is Newton's method, 
  stopping when the relative error in $L^2$-norm is less than $tol = 10^{-6}$. \\
 
 In order to check numerically the error estimates obtained in our theoretical analysis, we choose $k=10^{-5}$ and the numerical results with respect to the final time $T=0.001$ are listed in Tables~\ref{Tab1}-\ref{Tab3}. We can see that when $h\rightarrow 0$, $\Vert u(t_n) - u^n_h \Vert_{L^2 H^1}$ is convergent in optimal rate $\mathcal{O}(h)$, and $\Vert u^n_h - \mathcal{R}^u_h u^n_h\Vert_{L^2 H^1}$, $\Vert u(t_n) - u^n_h \Vert_{L^\infty L^2}$, $\Vert u^n_h - \mathcal{R}^u_h u^n_h \Vert_{L^\infty L^2}$, $\Vert v(t_n) - v^n_h \Vert_{L^\infty H^1}$ and $\Vert v^n_h - \mathcal{R}^v_h v^n_h \Vert_{L^\infty H^1}$ are convergent in optimal rate $\mathcal{O}(h^2)$. 

\begin{table}[h] 
\begin{footnotesize} 
\begin{center}
\begin{tabular}{|| c | c | c | c | c||}
\hline
\hline
$m \times m$ & $\Vert u(t_n) - u^n_h \Vert_{l^\infty L^2}$ & Order & $\Vert u^n_h - \mathcal{R}^u_h u^n_h \Vert_{l^\infty L^2}$ & Order  \\
\hline
 $40 \times 40$ & $2.5 \times 10^{-3}$  & -   & $1.5 \times 10^{-3}$  & - \\ 
\hline
 $50 \times 50$ &  $1.6 \times 10^{-3}$ & 1.9970 & $9 \times 10^{-4}$ & 1.9846 \\      
\hline
 $60 \times 60$  &  $1.1 \times 10^{-3}$   & 1.9980 & $7 \times 10^{-4}$ & 1.9896 \\       
\hline
 $70 \times 70$  &  $8 \times 10^{-4}$ & 1.9985 & $5 \times 10^{-4}$ & 1.9923 \\
\hline
 $80 \times 80$  & $6 \times 10^{-4}$   & 1.9989  &  $4 \times 10^{-4}$  & 1.9938 \\    
\hline
\hline
\end{tabular}
\end{center}
\caption{Error orders for $\Vert u(t_n) - u^n_h \Vert_{l^\infty L^2}$ and $\Vert u^n_h - \mathcal{R}^u_h u^n_h \Vert_{l^\infty L^2}$.} 
\label{Tab1} 
\end{footnotesize} 
\end{table}

\begin{table}[h] 
\begin{footnotesize} 
\begin{center}
\begin{tabular}{|| c | c | c | c | c||}
\hline
\hline
$m \times m$ & $\Vert u(t_n) - u^n_h \Vert_{l^2 H^1}$ & Order & $\Vert u^n_h - \mathcal{R}^u_h u^n_h \Vert_{l^2 H^1}$ & Order  \\
\hline
 $40 \times 40$ & $1.11 \times 10^{-2}$ & -   &  $5.219 \times 10^{-4}$  & - \\ 
\hline
 $50 \times 50$ & $8.9 \times 10^{-3}$ & 0.9978 &  $3.348 \times 10^{-4}$ & 1.9896 \\      
\hline
 $60 \times 60$  & $7.4 \times 10^{-3}$   & 0.9985 &  $2.328 \times 10^{-4}$& 1.9937 \\       
\hline
 $70 \times 70$  & $6.3 \times 10^{-3}$ & 0.9989 &  $1.711 \times 10^{-4}$ & 1.9966 \\
\hline
 $80 \times 80$  & $5.5 \times 10^{-3}$   & 0.9992  &  $1.310 \times 10^{-4}$  & 1.9988 \\    
\hline
\hline
\end{tabular}
\end{center}
\caption{Error orders for $\Vert u(t_n) - u^n_h \Vert_{l^2 H^1}$ and $\Vert u^n_h - \mathcal{R}^u_h u^n_h \Vert_{l^2 H^1}$.} 
\label{Tab2} 
\end{footnotesize} 
\end{table} 

\begin{table}[h] 
\begin{footnotesize} 
\begin{center}
\begin{tabular}{|| c | c | c | c | c||}
\hline
\hline
$m \times m$ & $\Vert v(t_n) - v^n_h \Vert_{l^\infty H^1}$ & Order & $\Vert v^n_h - \mathcal{R}^v_h v^n_h \Vert_{l^\infty H^1}$ & Order  \\
\hline
 $40 \times 40$ & $1.08 \times 10^{-2}$ & -   & $9.875 \times 10^{-4}$  & - \\ 
\hline
 $50 \times 50$ & $6.9 \times 10^{-3}$ & 1.9985 & $5.526 \times 10^{-4}$ & 2.6014 \\      
\hline
 $60 \times 60$  & $4.8 \times 10^{-3}$   & 1.9990 & $3.448 \times 10^{-4}$ & 2.5874 \\       
\hline
 $70 \times 70$  & $3.5 \times 10^{-3}$ & 1.9993 & $2.318 \times 10^{-4}$ & 2.5768 \\
\hline
 $80 \times 80$  & $2.7 \times 10^{-3}$   & 1.9995  &  $1.645 \times 10^{-4}$  & 2.5684 \\    
\hline
\hline
\end{tabular}
\end{center}
\caption{Error orders for $\Vert v(t_n) - v^n_h \Vert_{l^\infty H^1}$ and $\Vert v^n_h - \mathcal{R}^v_h v^n_h \Vert_{l^\infty H^1}$.} 
\label{Tab3} 
\end{footnotesize} 
\end{table} 

\section*{Acknowledgements}
The authors have been partially supported by MINECO grant MTM2015-69875-P
(Ministerio de Econom\'{\i}a y Competitividad, Spain) with the participation of FEDER.
The third author have also been supported by Vicerrector\'ia de Investigaci\'on y Extensi\'on of Universidad
Industrial de Santander.


\begin{thebibliography}{99}
\scriptsize
\bibitem{Nour}  \textrm{C.\ Amrouche and N.E.H.\ Seloula,} $L^p$-theory for vector potentials and Sobolev's inequalities for vector fields: application to the Stokes equations with pressure boundary conditions.
{\it Math. Models Methods Appl. Sci.} \textbf{23} (2013), no. 1, 37--92. 

\bibitem{C5:BJ} M. Bessemoulin-Chatard and A. J\"ungel, A finite volume scheme for a Keller-Segel model with additional cross-diffusion. {\it IMA J. Numer. Anal.} \textbf{34} (2014), no. 1, 96--122.

\bibitem{Brenner} S.C. Brenner and L.R. Scott, {\it The mathematical theory of finite element methods}, Third edition, Texts in Applied Mathematics, 15. Springer, New York (2008). 

\bibitem{CST} G. Chamoun, M. Saad and R. Talhouk, Monotone combined edge finite volume-finite element scheme for anisotropic Keller-Segel model. {\it Numer. Methods Partial Differential Equations} \textbf{30} (2014), no. 3, 1030--1065.

\bibitem{C5:Cristian} T. Cieslak, P. Laurencot and C. Morales-Rodrigo, Global existence and convergence to steady states in a chemorepulsion system. \emph{Parabolic and Navier-Stokes equations}. Part 1,  105-117, Banach Center Publ., 81, Part 1, Polish Acad. Sci. Inst. Math., Warsaw, 2008.

\bibitem{Eps} Y. Epshteyn and A. Izmirlioglu, Fully discrete analysis of a discontinuous finite element method for the Keller-Segel chemotaxis model. {\it J. Sci. Comput.} \textbf{40} (2009), no. 1-3, 211--256.

\bibitem{Filbet} F. Filbet, A finite volume scheme for the Patlak-Keller-Segel chemotaxis model. {\it Numer. Math.} \textbf{104} (2006), no. 4, 457--488. 

\bibitem{Fou}  F. Foucher, M. Ibrahim and M. Saad, Convergence of a positive nonlinear control volume finite element scheme for solving an anisotropic degenerate breast cancer development model. {\it Comput. Math. Appl.} \textbf{76} (2018), no. 3, 551-578.

\bibitem{Marcel} M. Freitag, Global existence and boundedness in a chemorepulsion system with superlinear diffusion. {\it Discrete Contin. Dyn. Syst.} \textbf{38} (2018), no. 11, 5943--5961.

\bibitem{FMD} F. Guill\'en-Gonz\'alez, M.A. Rodr\'iguez-Bellido and D.A. Rueda-G\'omez, Study of a chemo-repulsion model with quadratic production. Part I: Analysis of the continuous problem and time-discrete numerical schemes. (Submitted).

\bibitem{C5:FMD4} F. Guill\'en-Gonz\'alez, M.A. Rodr\'iguez-Bellido and D.A. Rueda-G\'omez, Unconditionally energy stable fully discrete schemes for a chemo-repulsion model. \emph{Mathematics of Computation} \textbf{88} (2019), no. 319, 2069--2099. 



\bibitem{Yulin} Y. Lai and Y. Xiao, Existence and asymptotic behavior of global solutions to chemorepulsion systems with nonlinear sensitivity. {\it Electron. J. Differential Equations} (2017), No. 254, 9 pp.

\bibitem{Marrocco} A. Marrocco, Numerical simulation of chemotactic bacteria aggregation via mixed finite elements. \emph{M2AN Math. Model. Numer. Anal.} \textbf{37} (2003), no. 4, 617--630.

\bibitem{necas} J. Necas, {\sl Les M\'ethodes Directes en Th\'eorie des Equations Elliptiques}. Editeurs Academia, Prague (1967).

\bibitem{Saito2} N. Saito, Error analysis of a conservative finite-element approximation for the Keller-Segel system of chemotaxis. {\it Commun. Pure Appl. Anal.} \textbf{11} (2012), no. 1, 339--364. 

\bibitem{Shen} J. Shen, Long time stability and convergence for fully discrete nonlinear Galerkin methods. \emph{Appl. Anal.} \textbf{38} (1990), 201--229.


\bibitem{T1} Y. Tao, Global dynamics in a higher-dimensional repulsion chemotaxis model with nonlinear sensitivity. {\it Discrete Contin. Dyn. Syst. Ser. B} \textbf{18} (2013), no. 10, 2705--2722.

\bibitem{Tel} J. Tello and D. Wrzosek, Inter-species competition and chemorepulsion. {\it J. Math. Anal. Appl.} \textbf{459} (2018), no. 2, 1233--1250.

\bibitem{Win1} M. Winkler, A critical blow-up exponent in a chemotaxis system with nonlinear signal production. \emph{Nonlinearity} \textbf{31} (2018), no. 5, 2031--2056.

\bibitem{Zhan-Zhu}J. Zhang, J. Zhu and R. Zhang, Characteristic splitting mixed finite element analysis of Keller-Segel chemotaxis models. {\it Applied Mathematics and Computation} \textbf{278} (2016) 33-44. 


\bibitem{Saito1}   G. Zhou and N. Saito, Finite volume methods for a Keller-Segel system: discrete energy, error estimates and numerical blow-up analysis. {\it Numer. Math.} \textbf{135} (2017), no. 1, 265-311.

\end{thebibliography}
\end{document}